\documentclass[12pt]{amsart}
\usepackage{amssymb,latexsym}
\usepackage{enumerate}

\makeatletter
\@namedef{subjclassname@2010}{%
  \textup{2010} Mathematics Subject Classification}
\makeatother

\numberwithin{equation}{section}

\newtheorem{theorem}{Theorem}[section]
\newtheorem{lemma}[theorem]{Lemma}
\newtheorem{prop}[theorem]{Proposition}
\newtheorem{cor}[theorem]{Corollary}

\theoremstyle{remark}
\newtheorem{definition}[theorem]{Definition}
\newtheorem{example}[theorem]{Example}
\newtheorem{remark}[theorem]{Remark}
\newtheorem*{xrem}{Remark}

\newcommand{\bC}{\mathbb{C}}
\newcommand{\bN}{\mathbb{N}}
\newcommand{\bQ}{\mathbb{Q}}
\newcommand{\bP}{\mathbb{P}}
\newcommand{\bR}{\mathbb{R}}
\newcommand{\bZ}{\mathbb{Z}}
\newcommand{\cC}{\mathcal{C}}

\newcommand{\cO}{\mathcal{O}}
\newcommand{\cP}{\mathcal{P}}
\newcommand{\cR}{\mathcal{R}}
\newcommand{\cS}{\mathcal{S}}
\newcommand{\cT}{\mathcal{T}}
\newcommand{\cV}{\mathcal{V}}

\newcommand{\Cli}{\cC^{\textit{li}}}

\newcommand{\et}{\quad\text{and}\quad}

\newcommand{\oF}{\overline{F}}
\newcommand{\oG}{\overline{G}}
\newcommand{\oT}{\overline{T}}
\newcommand{\Qbar}{\bar{\bQ}}
\newcommand{\ssi}{\quad\Longleftrightarrow\quad}

\newcommand{\un}{\mathbf{n}}
\newcommand{\uu}{\mathbf{u}}

\newcommand{\ux}{\mathbf{x}}
\newcommand{\uy}{\mathbf{y}}
\newcommand{\uz}{\mathbf{z}}

\begin{document}

\baselineskip=15pt 

\title[Simultaneous rational approximation]
{Simultaneous approximation to a real number and to its cube by rational numbers}
\author[S. Lozier]{St\'ephane Lozier}
\author[D. Roy]{Damien ROY}
\address{
   D\'epartement de Math\'ematiques\\
   Universit\'e d'Ottawa\\
   585 King Edward\\
   Ottawa, Ontario K1N 6N5, Canada}
\email[St\'ephane Lozier]{slozi062@uottawa.ca}
\email[Damien Roy]{droy@uottawa.ca}
\thanks{Research partially supported by NSERC}

\begin{abstract}
It is known that, for each real number $\xi$ such that $1,\xi,\xi^2$ are linearly independent over $\bQ$, the uniform exponent of simultaneous approximation to $(1,\xi,\xi^2)$ by rational numbers is at most $(\sqrt{5}-1)/2 \cong 0.618$ and that this upper bound is best possible.  In this paper, we study the analogous problem for $\bQ$-linearly independent triples $(1,\xi,\xi^3)$, and show that, for these, the uniform exponent of simultaneous approximation by rational numbers is at most $2(9+\sqrt{11})/35 \cong 0.7038$.  We also establish general properties of the sequence of minimal points attached to such triples that are valid for smaller values of the exponent.
\end{abstract}


\maketitle

\section{Introduction}
\label{sec:intro}

In order to construct approximations to real numbers by algebraic integers of bounded degree, H.~Davenport and W.~M.~Schmidt were led to study, through a duality argument, the problem of uniform approximation by rational numbers to consecutive powers of real numbers \cite{DSb}.  To describe their result, although in a slightly weaker form, fix a positive integer $n$ and a point $\Xi=(\xi_0,\dots,\xi_n) \in \bR^{n+1}$ with $\xi_0\neq 0$.  We say that a real number $\lambda\ge 0$ is a \emph{uniform exponent of approximation} to $\Xi$ (by rational numbers) if there exists a constant $c=c(\Xi)>0$ such that the system of inequations
\[
 |x_0| \le X
 \et
 \max_{1\le i\le n} |x_0\xi_i-x_i\xi_0| \le cX^{-\lambda}
\]
admits a non-zero solution $\ux=(x_0,\dots,x_n)\in\bZ^{n+1}$ for each real number $X\ge 1$.  Let $\lambda(\Xi)$ denote the supremum of these exponents $\lambda$.  Then, Theorems 1a, 2a and 4a of \cite{DSb} can essentially be summarized as follows:

\begin{theorem}[Davenport and Schmidt, 1969]
Let $n\ge 2$ be an integer and let $\xi\in\bR$ such that the point $\Xi=(1,\xi,\dots,\xi^n) \in \bR^{n+1}$ has $\bQ$-linearly independent coordinates.  Then, we have
\begin{equation}
 \label{intro:eq:DS}
 \lambda(\Xi)
 \le
 \begin{cases}
 1/\gamma \cong 0.618 &\text{if $n=2$,}\\
  1/2 &\text{if $n=3$,}\\
  [n/2]^{-1} &\text{if $n\ge 4$,}
 \end{cases}
\end{equation}
where $\gamma=(1+\sqrt{5})/2$ denotes the golden ratio, and $[n/2]$ stands for the integer part of $n/2$.
\end{theorem}

The problem remains to determine, for each $n\ge 1$, the supremum $\lambda_n$ of $\lambda(1,\xi,\dots,\xi^n)$ as $\xi$ runs through all real numbers which are not algebraic over $\bQ$ of degree $\le n$.  By Dirichlet's theorem on simultaneous approximation \cite[Ch.~II, Thm.~1A]{Sc}, we know that $\lambda_n\ge 1/n$ for each $n\ge 1$.  When $n=1$, this estimate is sharp: we have $\lambda_1=1$ since $\lambda(1,\xi)=1$ for each $\xi\in\bR\setminus\bQ$.  However, it is shown in \cite{Rb} that $\lambda_2=1/\gamma >1/2$.  So, \eqref{intro:eq:DS} is optimal for $n=2$.  For larger integers $n$, the value of $\lambda_n$ is unknown, but there have been some recent improvements upon \eqref{intro:eq:DS}.  In \cite{La}, M.~Laurent proved that $\lambda_n\le \lceil n/2 \rceil^{-1}$ for each $n\ge 3$, where $\lceil n/2 \rceil$ denotes the smallest integer greater than or equal to $n/2$.  Moreover, it is shown in \cite{Rc} that $\lambda_3\le (1 + 2\gamma - \sqrt{1+4\gamma^2})/2 \cong 0.4245$.

One goal of the present paper is to prove the following result of similar nature.

\begin{theorem}
 \label{intro:thm:main}
Let $\xi\in\bR$ such that $1,\xi,\xi^3$ are linearly independent over $\bQ$.  Then, we have
\[
 \lambda(1,\xi,\xi^3) \le \mu := \frac{2(9+\sqrt{11})}{35} \cong 0.7038.
\]
\end{theorem}

This estimate refines the upper bound $\lambda(1,\xi,\xi^3) \le 5/7 \cong 0.714$ established by the first author in \cite[Thm.~10.5]{Lo}, but it is not best possible neither.  The method that we present in this paper is capable of lowering it, possibly down to $(1+3\sqrt{5})/11 \cong 0.7007$ but we have not been able to go so far.

Before saying a word on this method, we mention two ``generic'' consequences of Theorem \ref{intro:thm:main}.  The first one follows from a simple adaptation of the arguments of Davenport and Schmidt in \cite[\S2]{DSb}.  Upon defining the \emph{height} $H(\alpha)$ of an algebraic number $\alpha$ as the largest absolute value of the coefficients of its irreducible polynomial in $\bZ[T]$, it reads as follows.

\begin{cor}
Let $\xi\in\bR$ such that $1,\xi,\xi^3$ are linearly independent over $\bQ$ and let $\tau < 1+1/\mu \cong 2.421$.  Then, there exists infinitely many algebraic integers $\alpha$ which are roots of polynomials of the form $T^4+aT^3+bT+c$ in $\bZ[T]$ and satisfy $|\xi-\alpha|\le H(\alpha)^{-\tau}$.
\end{cor}

The second consequence is a version of Gel'fond's transcendence criterion for lacunary polynomials.  It follows from a direct application of Jarn\'{\i}k's transference principle \cite[Thm.~1]{Ja}, and is in fact equivalent to Theorem \ref{intro:thm:main}.

\begin{cor}
Let $\xi\in\bR$ and let $\tau > 1/(1-\mu) \cong 3.376$.  Suppose that, for each sufficiently large real number $X$, there exists a non-zero polynomial $P(T) = aT^3+bT+c \in \bZ[T]$ with $\max\{|a|,|b|,|c|\} \le X$ and $|P(\xi)|\le X^{-\tau}$.  Then $1,\xi,\xi^3$ are linearly dependent over $\bQ$.
\end{cor}

The search of an optimal upper bound for the values $\lambda(1,\xi,\xi^3)$ from Theorem \ref{intro:thm:main} fits in the following general framework.  Let $\cC$ be a closed algebraic subset of $\bP^n(\bR)$ of dimension one defined by homogeneous polynomials of $\bQ[x_0,\dots,x_n]$, and let $\Cli$ denote the set of points $P$ of $\cC$ whose representatives $\Xi = (\xi_0,\dots,\xi_n) \in \bR^{n+1}$ have  $\bQ$-linearly independent coordinates.  Since $\lambda(a\Xi)=\lambda(\Xi)$ for each $a\in\bR^*$, we may define $\lambda(P)=\lambda(\Xi)$ independently of the choice of $\Xi$.  Then, the question is to determine the least upper bound $\lambda(\cC)$ of the numbers $\lambda(P)$ with $P\in\Cli$.

For example, for a fixed integer $k\ge 2$, let $\cC_{1,k}$ denote the zero locus of the polynomial $x_0^{k-1}x_2-x_1^k$ in $\bP^2(\bR)$.  Then $\Cli_{1,k}$ consists of the points of $\bP^2(\bR)$ having a set of $\bQ$-linearly independent homogeneous coordinates of the form $(1,\xi,\xi^k)$, and so $\lambda(\cC_{1,k})$ is the supremum of the numbers $\lambda(1,\xi,\xi^k)$ with $\xi\in\bR$ and $1,\xi,\xi^k$ linearly independent over $\bQ$.  Then, for $k=2$, we have $\lambda(\cC_{1,2})=1/\gamma\cong 0.618$ by \cite[Thm.~1.1]{Rb}, while for $k=3$, the above Theorem \ref{intro:thm:main} gives $\lambda(\cC_{1,3})\le \mu \cong 0.7038$. In this context, it would be interesting to know if there exist curves $\cC$ for which $\lambda(\cC)$ is arbitrarily close to 1.

The proof of Theorem \ref{intro:thm:main} goes first by attaching to the triple $(1,\xi,\xi^3)$ a sequence of minimal points $(\ux_i)_{i\ge 1}$ from $\bZ^3$, as in \cite{DSb}.  A simple but crucial property of this sequence is that $\ux_{i-1}$, $\ux_i$ and $\ux_{i+1}$ are linearly independent for infinitely many indices $i\ge 2$.  For such $i$, let $j$ be the next integer with the same property. Then, the points $\ux_i,\ux_{i+1},\dots,\ux_j$ all lie in the same $2$-dimensional subspace of $\bR^3$.  Initially and for a long time, we tried to construct explicit auxiliary polynomials $P$ with integer coefficients vanishing at triples or even quadruples of these points, including points coming before $\ux_i$ or after $\ux_j$, but this soon became very complicated.  We will not mention these constructions here (except for the polynomial $g$ in Section \ref{sec:F}) because we discovered that it is in fact much more efficient to deal simply with the pairs $(\ux_i,\ux_j)$, provided that we take into account the content of their cross products $\ux_i\wedge\ux_j$, namely the gcd of its coordinates, denoted $|q_i|$ for an integer $q_i$ defined in Section \ref{sec:search}.  Assuming a lower bound $\lambda(1,\xi,\xi^3)>\lambda_0$, the idea is to construct polynomials $P\in\bZ[\ux,\uy]$ for which the integer $|P(\ux_i,\ux_j)|$ is relatively small for analytic reasons, and divisible by a certain power $q_i^k$ of $q_i$ for algebraic reasons.  If it happens that $|P(\ux_i,\ux_j)|<|q_i|^k$, then we conclude that $P(\ux_i,\ux_j)=0$.  On the other hand, if we can show that $P(\ux_i,\ux_j)\neq 0$ by some arithmetic argument, then we obtain $|q_i|^k \le |P(\ux_i,\ux_j)|$ which imposes constrains on the growth of the points $\ux_i$ and $\ux_j$.  The details concerning the construction of such polynomials are explained in Section \ref{sec:search}.

The most basic polynomial in this respect is $\varphi(\ux)=x_0^2x_2-x_1^3$, which defines the curve $\cC_{1,3}$.  In Section \ref{sec:prelim}, we show that, if $\lambda(1,\xi,\xi^3) > 2/3$, then $\varphi(\ux_i)\neq 0$ for each sufficiently large $i$.  Then, imitating the proof of Theorem 1a from \cite{DSb}, we conclude, as a first approximation, that $\lambda(1,\xi,\xi^3)\le \sqrt{3}-1 \cong 0.732$.  The next most important polynomial is $F$ introduced and studied in Section \ref{sec:F}. Then come $D^{(2)}$, $D^{(3)}$ and $D^{(6)}$ introduced in Section \ref{sec:search}, with the property that $D^{(k)}(\ux_i,\ux_j)$ is divisible by $q_i^k$ for each pair $(i,j)$ as above and $k=2,3,6$. They are the simplest polynomials that we found.  Assuming $\lambda(1,\xi,\xi^3)>0.6985$, it appears that, for $i$ large enough, none of them vanishes at the point $(\ux_i,\ux_j)$.  This is proved for $F$ in Section \ref{sec:NVF}, for $D^{(2)}$ in Section \ref{sec:NVD2}, and for both $D^{(3)}$ and $D^{(6)}$ in Section \ref{sec:NVD3D6}.  Then, the lower bound for $\lambda(1,\xi,\xi^3)$ given by Theorem \ref{intro:thm:main} is proved in Section \ref{sec:NVD3D6} on the basis of these non-vanishing results.

In Section \ref{sec:pol}, we show that, if $\lambda(1,\xi,\xi^3) > (1+3\sqrt{5})/11 \cong 0.7007$, then there exists a non-zero polynomial $P$ which vanishes at $(\ux_i,\ux_j)$ for infinitely many pairs $(i,j)$ as above.  This non-explicit construction suggests that we probably have $\lambda(1,\xi,\xi^3) \le (1+3\sqrt{5})/11$ because, if such a polynomial relation exists, we would expect it to be relatively simple, but we already ruled out the simplest ones.

All polynomials that we construct come from a graded factorial ring $\cR \subset \bQ[\ux,\uy]$ defined in Section \ref{sec:search}.  When $\lambda(1,\xi,\xi^3)>2/3$, it is shown in Section \ref{sec:search:remark} that any two relatively prime homogeneous elements of $\cR$ have only finitely many common zeros of the form $(\ux_i,\ux_j)$.  This suggests a natural way to avoid the delicate non-vanishing results.  In Section \ref{sec:special}, we give an example where this strategy applies.  However, in the general situation of Section \ref{sec:pol} we have not been able to put it in practice.

%
%

\section{Preliminaries}
\label{sec:prelim}

From now on, we fix a real number $\xi$ such that $1,\xi,\xi^3$ are linearly independent over $\bQ$.  For each $\ux=(x_0,x_1,x_2)\in\bR^3$, we set
\[
 \|\ux\| = \max\{|x_0|,|x_1|,|x_2|\},
 \quad
 L(\ux) = L_\xi(\ux) = \max\{|x_1-x_0\xi|,|x_2-x_0\xi^3|\}.
\]
We also fix choices of $\lambda>0$ and $c>0$ such that, for each sufficiently large real number $X$, there exists a non-zero point $\ux\in\bZ^3$ with
\begin{equation}
 \label{prelim:mainhyp}
 \|\ux\| \le X \et L(\ux) \le c X^{-\lambda}.
\end{equation}
Our goal is to show that $\lambda\le 2(9+\sqrt{11})/35$.  We proceed in several steps.  In what follows, whenever we use the Vinogradov symbols $\gg$, $\ll$, or their conjunction $\asymp$, the implied constants depend only on $\xi$, $\lambda$ and $c$.

We first note that an argument similar to that of Davenport and Schmidt in \cite[\S 3]{DSb} shows the existence of a sequence of non-zero points $(\ux_i)_{i\ge 1}$ in $\bZ^3$ such that, upon writing
\[
 X_i = \|\ux_i\| \et L_i = L(\ux_i)
\]
for each $i\ge 1$, we have
\begin{itemize}
 \item[a)] $X_1 < X_2 < X_3 < \cdots$,
 \item[b)] $1/2 > L_1 > L_2 > L_3 > \cdots$,
 \item[c)] if $L(\ux) < L_i$ for some $\ux\in\bZ\setminus\{0\}$ and some $i\ge 1$, then $\|\ux\|\ge X_{i+1}$.
\end{itemize}
Such a sequence is unique up to its first terms and up to multiplication of each of its terms by $\pm 1$ because, for $\ux,\uy\in\bZ$ with $L(\ux)=L(\uy)<1/2$, we have $\uy=\pm \ux$.  We call it a sequence of \emph{minimal points} for the triple $(1,\xi,\xi^3)$.  The construction of Davenport and Schmidt is slightly different in that they use the absolute value of the first coordinate of each point instead of its norm, but it can be checked that the resulting sequences are the same modulo the above equivalence relation.

Fix $(\ux_i)_{i\ge 1}$, $(X_i)_{i\ge 1}$ and $(L_i)_{i\ge 1}$ as above.  Then, our main hypothesis \eqref{prelim:mainhyp} is equivalent to
\begin{equation}
 \label{prelim:hypLiXip}
 L_i \le c X_{i+1}^{-\lambda}
\end{equation}
for each sufficiently large $i$.  Moreover, the points $\ux_i$ are primitive and so, they are two by two linearly independent over $\bQ$.  We write their coordinates in the form
\[
 \ux_i = (x_{i,0},x_{i,1},x_{i,2})\,.
\]
Then, for each $i\ge 1$, the condition $L_i<1/2$ from b) implies that $x_{i,0}\neq 0$.

\begin{definition}
For each $i\ge 1$, we denote by $W_i = \langle \ux_i,\ux_{i+1}\rangle_\bR$ the two-dimensional subspace of $\bR^3$ generated by $\ux_i$ and $\ux_{i+1}$.
\end{definition}

Following \cite[Ch.~I, \S 4]{Sc}, when $W$ is a subspace of $\bR^n$ generated by elements of $\bQ^n$, we define its \emph{height} $H(W)$ as the covolume (or determinant) of the lattice $W\cap \bZ^n$ in $W$.  The next lemma is a first step in estimating the height of the subspaces $W_i$ of $\bR^3$.

\begin{lemma}
\label{prelim:lemma:HWi}
For each $i\ge 1$, the set $\{\ux_i,\ux_{i+1}\}$ is a basis of $W_i\cap \bZ^3$ and we have $H(W_i) \asymp \|\ux_i\wedge\ux_{i+1}\|$.
\end{lemma}

The first assertion follows by a direct adaptation of the arguments of Davenport and Schmidt in their proof of \cite[Lemma 2]{DSa} (see \cite[Lemma 4.1]{Rb}).  It implies that the cross product $\ux_i\wedge\ux_{i+1}$ is a primitive point of $\bZ^3$ and that $H(W_i)$ is its Euclidean norm.

The proof of the next lemma is similar to that of \cite[Lemma 4.1]{Rb} but requires some adjustments as its scope is more general.

\begin{lemma}
\label{prelim:lemma:xjwedgexi}
For any integers $1\le i<j$, we have $\|\ux_i\wedge\ux_j\| \asymp X_jL_i$.
\end{lemma}

\begin{proof}
The estimate $\|\ux_i\wedge\ux_j\| \ll X_jL_i$ is a well-known fact which follows, for example, from \cite[Lemma 3.1(i)]{Rb}.  We claim that, for $1\le i<j$ with $j$ large enough, we have $\|\ux_i\wedge\ux_j\| \ge |x_{j,0}|L_i/3$.  As $|x_{j,0}|\asymp X_j$, this will suffice to complete the proof of the lemma.  To prove this claim, we set $\Xi=(1,\xi,\xi^3)$ and define $\Delta_\ell=\ux_\ell-x_{\ell,0}\Xi$ for each $\ell\ge 1$.  We also assume, as we may, that $x_{\ell,0}> 0$ for each $\ell\ge 1$.  Then, the last two coordinates of $x_{i,0}\Delta_j-x_{j,0}\Delta_i$ coincide, up to sign, with coordinates of $\ux_i\wedge\ux_j$ and so we have
\begin{equation}
 \label{prelim:lemma:xjwedgexi:eq}
 \|\ux_i\wedge\ux_j\| \ge \|x_{i,0}\Delta_j-x_{j,0}\Delta_i\|.
\end{equation}
Now, suppose that $\|\ux_i\wedge\ux_j\| < x_{j,0}L_i/3$.  Since $\|\Delta_\ell\| = L_\ell$ for each $\ell\ge 1$, we deduce from \eqref{prelim:lemma:xjwedgexi:eq} that $x_{i,0}L_j > (2/3)x_{j,0}L_i$, and so $x_{j,0}< (3/2)x_{i,0}$.  Assuming that $j$ is large enough, this implies that $\|\ux_j-\ux_i\| < \|\ux_i\|$.  Then, since $\ux_i$ is a minimal point, we conclude that $L_i < L(\ux_j-\ux_i) = \|\Delta_j-\Delta_i\|$ and the inequality \eqref{prelim:lemma:xjwedgexi:eq} yields
\[
 x_{j,0}L_i/3
   > \| x_{i,0}(\Delta_j-\Delta_i) - (x_{j,0}-x_{i,0})\Delta_i \|
   > x_{i,0}L_i - (x_{j,0}-x_{i,0})L_i
\]
in contradiction with $x_{j,0}< (3/2)x_{i,0}$.  Thus our hypothesis forces $j$ to be bounded.
\end{proof}

Combining the two previous results, we get:

\begin{cor}
\label{prelim:cor:HWi}
For any $i\ge 1$, we have $H(W_i) \asymp X_{i+1}L_i$.
\end{cor}

\begin{definition}
We denote by $I$ the set of all integers $i\ge 2$ such that $\ux_{i-1},\ux_i,\ux_{i+1}$ are linearly independent.
\end{definition}

The same argument as in the proof of \cite[Lemma 5]{DSb} shows that $I$ is an infinite set.  We order its elements by increasing magnitude.

A result of Schmidt \cite[Ch.~I, Lemma 8A]{Sc} shows that $H(S\cap T)H(S+T) \le H(S)H(T)$ for any pair of subspaces $S$ and $T$ of $\bR^n$ generated by elements of $\bQ^n$.  Applying it to the present situation, like in \cite[\S 3]{Rc}, we obtain:

\begin{prop}
\label{prelim:prop:height}
For any pair of consecutive elements $i<j$ in $I$, we have
\[
 X_j \ll (X_{i+1}X_{j+1})^{1-\lambda}
 \et
 1 \ll X_{j+1}L_jL_{j-1}.
\]
\end{prop}

\begin{proof}
For such $i$ and $j$, we have $W_i\cap W_j = \langle \ux_j\rangle_\bR$ and $W_i+W_j = \bR^3$. Since $H(\langle \ux_j \rangle_\bR)$ is the Euclidean norm of $\ux_j$, we deduce from the inequality of Schmidt recalled above that $X_j \ll H(W_i) H(W_j)$.  The conclusion then follows from Corollary \ref{prelim:cor:HWi} using $L_i\ll X_{i+1}^{-\lambda}$ and $L_j\ll X_{j+1}^{-\lambda}$ to get the first estimate, and using $W_i = W_{j-1}$ to get the second one.
\end{proof}

An alternative proof of the second estimate goes by observing that the determinant of the three points $\ux_{j-1},\ux_j,\ux_{j+1}$ is non-zero and by estimating from above its absolute value as in \cite[Lemma 4]{DSb}.

\begin{cor}
\label{prelim:cor:approximable}
We have $X_j^{\lambda/(1-\lambda)} \ll X_{j+1}$ and $L_j \ll X_j^{-\lambda^2/(1-\lambda)}$ for any $j\in I$.  In particular, if $\xi$ or $\xi^3$ is badly approximable, then we must have $\lambda \le (\sqrt{5}-1)/2\cong 0.618$.
\end{cor}

\begin{proof}
We may assume that $j$ is not the first element of $I$.  Then, upon denoting by $i$ the preceding element of $I$, we have $X_{i+1}\le X_j$ and the first estimate of the proposition leads to $X_j^\lambda \ll X_{j+1}^{1-\lambda}$, thus $L_j \ll X_{j+1}^{-\lambda} \ll X_j^{-\lambda^2/(1-\lambda)}$.  This proves the first assertion.  Now, if $\xi$ or $\xi^3$ is badly approximable, then we also have $L_j\gg X_j^{-1}$ for each $j\ge 1$.  In particular, this holds for each $j\in I$, thus $\lambda^2/(1-\lambda) \le 1$ and so $\lambda \le (\sqrt{5}-1)/2$.
\end{proof}

\begin{cor}
\label{prelim:cor:XjLj}
Suppose that $\lambda\ge 2/3$.  Then, for each $j\in I$, we have $X_j^2\ll X_{j+1}$ and $L_j\ll L_{j-1}^2$.
\end{cor}

\begin{proof}
The first estimate follows directly from the preceding corollary.  Since $X_{j+1} \ll L_j^{-1/\lambda} \le L_j^{-3/2}$, the second estimate of Proposition \ref{prelim:prop:height} yields $1\ll L_j^{-1/2}L_{j-1}$ and so $L_j\ll L_{j-1}^2$.
\end{proof}

Up to now, all of the above applies not only to the triple $(1,\xi,\xi^3)$ but also to any $\bQ$-linearly independent triple of real numbers of the form $(1,\xi,\eta)$.  The polynomials that we now introduce are specific to the present study.

\begin{definition}
Let $\ux=(x_0,x_1,x_2)$, $\uy=(y_0,y_1,y_2)$ and $\uz=(z_0,z_1,z_2)$ be triples of indeterminates.  We set
\[
 \varphi(\ux) = x_0^2x_2-x_1^3
 \et
 \Phi(\ux,\uy,\uz) = x_0y_0z_2 + x_0y_2z_0 + x_2y_0z_0 - 3x_1y_1z_1.
\]
\end{definition}

The cubic form $\varphi$ satisfies $\varphi(1,\xi,\xi^3)=0$ and for that reason plays in the present context a role which is analog to that of the quadratic form $x_0x_2-x_1^2$ in \cite[\S 3]{DSb}.  The polynomial $\Phi$ is the symmetric trilinear form for which
\[
 \Phi(\ux,\ux,\ux) = 3 \varphi(\ux).
\]
The next result, analogous to Lemma 2 of \cite{DSb}, is the main result of this section.

\begin{theorem}
 \label{prelim:thm:phi}
Suppose that $\lambda>2/3$.  Then we have $\varphi(\ux_i)\neq 0$ for each sufficiently large index $i$.
\end{theorem}

\begin{proof}  Suppose that $\varphi(\ux_i)=0$ for some integer $i\ge 2$.  Then, $\ux_i$ takes the form $\ux_i=(p^3,p^2q,q^3)$ for some non-zero coprime integers $p$ and $q$.  The vector $\un = \ux_{i-1}\wedge\ux_i$ is non-zero and orthogonal to $\ux_i$. However, as $p^2q\neq 0$, the vector $\ux_i$ is not orthogonal to $(0,1,0)$. So the first or the third coordinate of $\un$ is non-zero.  As the first coordinate of $\un$ is an integer multiple of $q$ and the third an integral multiple of $p^2$, this implies that $\|\un\| \ge \min\{|q|,p^2\} \gg X_i^{1/3}$.  On the other hand, Lemma \ref{prelim:lemma:xjwedgexi} gives $\|\un\| \asymp X_i L_{i-1} \ll X_i^{1-\lambda}$.  Combining these estimates gives $X_i^{1-\lambda} \gg X_i^{1/3}$ and so $i$ is bounded from above.
\end{proof}

\begin{definition}
We set $\Xi = (1,\xi,\xi^3)$ and, for each $i\ge 1$, we put
\[
 \delta(\ux_i) = \Phi(\ux_i,\Xi,\Xi) = 2x_{i,0}\xi^3 - 3x_{i,1}\xi^2 + x_{i,2}
 \et
 \delta_i = |\delta(\ux_i)|.
\]
\end{definition}

These quantities $\delta(\ux_i)$ are useful in dealing with limited developments of polynomials involving the function $\Phi$ as the next result illustrates.

\begin{cor}
\label{prelim:cor:est_varphi}
Suppose that $\lambda >2/3$.  For each sufficiently large $i$, we have
\[
 |\varphi(\ux_i)| \asymp X_i^2\delta_i
 \et
 X_i^{-2} \ll \delta_i \ll L_i.
\]
\end{cor}

\begin{proof}
Put $\Delta_i=\ux_i-x_{i,0}\Xi$, so that $L_i=\|\Delta_i\|$.  By the multilinearity of $\Phi$ and the fact that $\Phi(\Xi,\Xi,\Xi)=3\varphi(\Xi)=0$, we find
\[
 \begin{aligned}
 \delta(\ux_i)
  &= \Phi(\Delta_i,\Xi,\Xi) = \cO(L_i),\\
 \varphi(\ux_i)
  &= x_{i,0}^2\Phi(\Delta_i,\Xi,\Xi) + x_{i,0}\Phi(\Delta_i,\Delta_i,\Xi) + \varphi(\Delta_i)
   = x_{i,0}^2\delta(\ux_i) + \cO(X_iL_i^2).
 \end{aligned}
\]
For all sufficiently large values of $i$, Theorem \ref{prelim:thm:phi} gives $\varphi(\ux_i)\neq 0$.  As $\varphi(\ux_i)$ is an integer and as $X_iL_i^2\ll X_i^{1-2\lambda}=o(1)$, we deduce that $|\varphi(\ux_i)| \asymp X_i^2\delta_i$ for each $i$ with $\varphi(\ux_i)\neq 0$ and so, $\delta_i\gg X_i^{-2}$ for the same values of $i$.
\end{proof}

\begin{cor}
 \label{prelim:cor:first_est_lambda}
We have $\lambda \le \sqrt{3}-1 \cong 0.732$.
\end{cor}

\begin{proof}
By Corollary \ref{prelim:cor:est_varphi}, we have $1 \ll X_j^2\delta_j \ll X_j^2L_j$ for each large enough index $j$.  As $L_j\ll X_{j+1}^{-\lambda}$, this gives $X_{j+1}^{\lambda} \ll X_j^2$.  When $j\in I$, Corollary \ref{prelim:cor:approximable} also gives $X_j^\lambda \ll X_{j+1}^{1-\lambda}$.  Combining the two estimates and letting $j$ go to infinity within $I$, we conclude that $\lambda^2\le 2(1-\lambda)$ and so $\lambda \le \sqrt{3}-1$.
\end{proof}

As we will see this upper bound is not optimal and we will improve it in what follows.  We end this section with a general estimate which will be useful for this purpose.

\begin{prop}
\label{prelim:prop:est_Phi}
Suppose that $\lambda>2/3$.  For any choice of integers $i\le j\le k$ with $i\in I$ large enough, we have $\Phi(\ux_i,\ux_j,\ux_k)\neq 0$ and
\[
 |\Phi(\ux_i,\ux_j,\ux_k)| \asymp X_jX_k\delta_i \asymp \frac{X_jX_k}{X_i^2} |\varphi(\ux_i)|.
\]
\end{prop}

\begin{proof}
In view of Corollary \ref{prelim:cor:est_varphi}, it suffices to show that $|\Phi(\ux_i,\ux_j,\ux_k)| \asymp X_jX_k\delta_i$ for all triples of integers $i\le j\le k$ with $i\in I$ large enough.  Using the multilinearity of the function $\Phi$ as in the proof of Corollary \ref{prelim:cor:est_varphi}, we find
\[
 \Phi(\ux_i,\ux_j,\ux_k)
   = x_{j,0}x_{k,0}\delta(\ux_i) + x_{i,0}x_{k,0}\delta(\ux_j) + x_{i,0}x_{j,0}\delta(\ux_k) + \cO(X_kL_iL_j),
\]
assuming only $1\le i\le j\le k$.  In the right hand side of this equality, the first three summands may not be distinct.  To conclude, we simply need to show that, when $i\in I$, we have $X_i\delta_j = o(X_j\delta_i)$ if $i<j$, $X_i\delta_k = o(X_k\delta_i)$ if $i<k$, and $L_iL_j = o(X_j\delta_i)$.  The estimates for $\delta_i,\delta_j,\delta_k$ provided by Corollary \ref{prelim:cor:est_varphi} reduce the problem to showing
\[
 X_iL_j = o(X_jX_i^{-2}) \ \text{if $i<j$,}
 \ \
 X_iL_k = o(X_kX_i^{-2}) \ \text{if $i<k$,}
 \ \
 L_iL_j = o(X_jX_i^{-2}).
\]
The third estimate is clear because $L_iL_j \le L_i^2 \ll X_i^{-4/3}$ and $X_jX_i^{-2} \ge X_i^{-1}$.  To prove the first estimate, we note that, by Corollary \ref{prelim:cor:XjLj}, we have $X_i^2 \ll X_{i+1}$ and $L_i\ll L_{i-1}^2 \ll X_i^{-4/3}$ since $\lambda>2/3$ and $i\in I$.  Thus, when $i<j$, we obtain $X_iL_j \le X_iL_i = o(1)$ while $X_jX_i^{-2} \ge X_{i+1}X_i^{-2} \gg 1$.  The proof of the second estimate is the same.
\end{proof}

%
%

\section{The polynomial $F$ and the point $\psi$}
\label{sec:F}

We introduce two new actors in the present study.

\begin{definition}
For triples of indeterminates $\ux$ and $\uy$, we set:
\begin{align*}
 F(\ux,\uy)
   &= \Phi(\ux,\ux,\uy)^2 - 4\varphi(\ux)\Phi(\ux,\uy,\uy),\\
  \psi(\ux,\uy)
   &= \Phi(\ux,\ux,\uy)\ux-2\varphi(\ux)\uy.
\end{align*}
\end{definition}

This section is devoted to estimates for $F$.  In particular, we will show that $|F(\ux_i,\ux_j)|$ is a relatively small integer for each pair of consecutive indices $i<j$ in $I$, and that, if it is non-zero for infinitely many such pairs, then $\lambda>5/7$.  We also provide estimates for $\psi$ which we view as an analog of the point $[\ux,\ux,\uy]$ defined in \cite[\S2]{Rb}, which plays a crucial role in the study \cite{Rb} of simultaneous approximation to a real number and to its square (see \cite[Cor.~5.2]{Rb}).  The polynomial $F$ however has no analog in \cite{Rb}.

Recalling that $\Xi=(1,\xi,\xi^3)$ and using the notation of Section 2 for the coordinates of points, we first establish the following formulas.

\begin{lemma}
\label{F:lemma:Fpsi}
We have the identities
\begin{itemize}
 \item[(i)] $F(\ux,\Xi) = F(\ux-x_0\Xi,\Xi)$,
 \item[(ii)] $\varphi(\psi(\ux,\uy)) = -\varphi(\ux)\Phi(\ux,\ux,\uy)F(\ux,\uy) - 8\varphi(\ux)^3\varphi(\uy)$.
\end{itemize}
\end{lemma}

Note that $\varphi(\psi(\ux,\uy))$ simplifies to $- 8\varphi(\ux)^3\varphi(\uy)$ for points $\ux,\uy\in\bZ^3$ with $F(\ux,\uy)=0$, a fact that is interesting to compare with \cite[Lemma 2.1(i)]{Rb}.

\begin{proof}
By the multilinearity of $\Phi$, we have, for indeterminates $a$ and $b$,
\[
 \varphi(a\ux+b\uy) = a^3\varphi(\ux) + a^2b\Phi(\ux,\ux,\uy) + ab^2\Phi(\ux,\uy,\uy) + b^3\varphi(\uy).
\]
Substituting $\Phi(\ux,\ux,\uy)$ for $a$ and $-2\varphi(\ux)$ for $b$ in this identity yields (ii). We also note that $F(\ux,\Xi)$ is the discriminant of
\[
 \varphi(\ux+T\Xi) = \varphi(\ux) + \Phi(\ux,\ux,\Xi)T + \Phi(\ux,\Xi,\Xi)T^2
\]
viewed as a polynomial in $T$.  Then (i) follows from the fact that the discriminant of a polynomial $p(T)$ stays invariant under the change of variable $T\mapsto T-x_0$.
\end{proof}

In the course of this research, we were also lead to work with polarized versions of $F$.  In particular the polynomial
\[
 g(\ux,\uu,\uy)
   = \Phi(\ux,\uu,\uy)\Phi(\ux,\ux,\uy) - \Phi(\ux,\ux,\uu)\Phi(\ux,\uy,\uy) - \varphi(\ux)\Phi(\uu,\uy,\uy),
\]
involving a third triple of indeterminates $\uu$, was playing a central role until we discovered the approach that will be presented in the next section.  We simply mention its existence in case it comes back in future investigations.

\begin{prop}
\label{F:prop:est_F}
For any $\ux,\uy\in\bZ^3$, we have
\[
 \begin{aligned}
  F(\ux,\uy) =
   &-4x_0^3y_0\delta(\ux)\delta(\uy) + \cO\big( \|\uy\|^2L(\ux)^4 + \|\ux\|^4L(\uy)^2 \big),\\
  \|\psi(\ux,\uy)\|
   &\ll \|\ux\|\/\|\uy\|\/L(\ux)^2 + \|\ux\|^3L(\uy).
 \end{aligned}
\]
\end{prop}

\begin{proof}
Write $\uy=y_0\Xi+\Delta\uy$.  Since $F$ is quadratic in its second argument, we find
\begin{equation}
 \label{F:prop:est_F:eq1}
 F(\ux,\uy) = y_0^2F(\ux,\Xi) + 2y_0A + F(\ux,\Delta\uy),
\end{equation}
where
\[
 A = \Phi(\ux,\ux,\Xi)\Phi(\ux,\ux,\Delta\uy)-4\varphi(\ux)\Phi(\ux,\Delta\uy,\Xi).
\]
To estimate $A$, we write $\ux=x_0\Xi+\Delta\ux$ and expand it as a polynomial in $x_0$.  Since
\[
 \begin{aligned}
 \Phi(\ux,\ux,\Xi) &= 2x_0\delta(\ux)+\cO(L(\ux)^2),\\
 \Phi(\ux,\ux,\Delta\uy) &= x_0^2\delta(\uy)+\cO(\|\ux\|\/L(\ux)L(\uy)),\\
 \varphi(\ux) &= x_0^2\delta(\ux)+\cO(\|\ux\|\/L(\ux)^2),\\
 \Phi(\ux,\Delta\uy,\Xi) &=x_0\delta(\uy)+\cO(L(\ux)L(\uy)),
 \end{aligned}
\]
we obtain
\[
 A = -2x_0^3\delta(\ux)\delta(\uy) + \cO(\|\ux\|^2L(\ux)^2L(\uy)).
\]
By Lemma \ref{F:lemma:Fpsi} (i), we also have $F(\ux,\Xi) = F(\Delta\ux,\Xi) = \cO(L(\ux)^4)$, while it is clear that $F(\ux,\Delta\uy) = \cO(\|\ux\|^4L(\uy)^2)$.  Substituting these estimates into \eqref{F:prop:est_F:eq1} yields
\begin{align*}
 F(\ux,\uy) = &-4x_0^3y_0\delta(\ux)\delta(\uy)\\
    &+ \cO\big(\|\uy\|^2L(\ux)^4 + \|\ux\|^2\|\uy\|L(\ux)^2L(\uy) + \|\ux\|^4L(\uy)^2\big).
\end{align*}
Finally, we may omit the middle term in the error estimate as it is the geometric mean of the other two.  The estimate for $\|\psi(\ux,\uy)\|$ is proved along similar lines and we leave this task to the reader.
\end{proof}

\begin{cor}
 \label{F:cor:est_Fij}
Suppose that $\lambda>2/3$.  For any pair of consecutive integers $i<j$ in $I$, we have
\[
 F(\ux_i,\ux_j)
   = - 4 x_{i,0}^3 x_{j,0} \delta(\ux_i) \delta(\ux_j) + \cO( X_j^2 L_i^4).
\]
\end{cor}

\begin{proof}
Since $\lambda>2/3$, Corollary \ref{prelim:cor:XjLj} gives $X_i^2\ll X_{i+1}\le X_j$ and $L_j \ll L_{j-1}^2 \le L_i^2$, so $X_i^4L_j^2 \ll X_j^2L_i^4$ and thus we may omit the product $X_i^4L_j^2$ in the error term from the preceding proposition.
\end{proof}

If we assume furthermore that $F(\ux_i,\ux_j)=0$, then the above estimate yields $X_i^3 X_j\delta_i \delta_j \ll X_j^2 L_i^4 \ll X_j^2 X_{i+1}^{-4\lambda}$ and, in view of Corollary \ref{prelim:cor:est_varphi}, this has the following consequence.

\begin{cor}
\label{F:cor:Fzero}
Suppose that $\lambda>2/3$.  Then, for any pair of consecutive elements $i<j$ in $I$ with $F(\ux_i,\ux_j) = 0$, we have
\[
 |\varphi(\ux_i)\varphi(\ux_j)| \ll X_i^{-1} X_{i+1}^{-4\lambda} X_j^3.
\]
\end{cor}

The next result deals with the complementary case where $F(\ux_i,\ux_j)\neq 0$.

\begin{prop}
\label{F:prop:Fnotzero}
Suppose that $\lambda>(5-\sqrt{13})/2 \cong 0.697$.  Then, for consecutive elements $i<j$ in $I$,
we have $|\varphi(\ux_i)\varphi(\ux_j)| = o(X_i^{-1}X_j)$ as $i\to\infty$.  For the pairs
$(i,j)$ with $F(\ux_i,\ux_j)\neq 0$, we also have
\[
 1 \le |F(\ux_i,\ux_j)| \ll X_j^2L_i^4 \et X_j \gg X_{i+1}^{2\lambda}.
\]
\end{prop}

\begin{proof}
We first note that the ratio $\theta=(1-\lambda)/\lambda$ satisfies $2\theta\ge \lambda$ since, by Corollary \ref{prelim:cor:first_est_lambda}, we have $\lambda\le \sqrt{3}-1$.  Then, for consecutive elements $i<j$ in $I$, we obtain
\begin{alignat}{2}
  \frac{X_i}{X_j}|\varphi(\ux_i)\varphi(\ux_j)|
   &\asymp X_i^3X_j\delta_i\delta_j
    \ll X_i^3 &&X_{i+1}^{-\lambda} X_j X_{j+1}^{-\lambda}
      \quad\text{by Corollary \ref{prelim:cor:est_varphi},} \notag\\
   &\ll X_{i+1}^{3\theta-\lambda} X_j X_{j+1}^{-\lambda}
      &&\ \text{since $X_i\ll X_{i+1}^\theta$ by Corollary \ref{prelim:cor:approximable},} \notag\\
   &\le X_j^{3\theta+1-\lambda} X_{j+1}^{-\lambda}
      &&\ \text{since $X_{i+1}\le X_j$ and $3\theta-\lambda>0$,} \notag\\
   &\ll X_{j+1}^{3\theta^2+\theta-1}
      &&\ \text{since $X_j\ll X_{j+1}^\theta$ by Corollary \ref{prelim:cor:approximable}.} \notag
\end{alignat}
As $3\theta^2+\theta-1 = (\lambda^2-5\lambda+3)/\lambda^2 <0$, this proves our first assertion.  It also shows that $X_i^3X_j\delta_i\delta_j=o(1)$ as $i\to \infty$.  Thus Corollary \ref{F:cor:est_Fij} yields $1\le |F(\ux_i,\ux_j)| \ll X_j^2L_i^4$ when the integer $F(\ux_i,\ux_j)$ is non-zero and $i$ is large enough.  In that case, using $L_i\ll X_{i+1}^{-\lambda}$, we find $X_j \gg X_{i+1}^{2\lambda}$.
\end{proof}

The non-vanishing of $F$ has important consequences.  The next lemma provides useful estimates that we will need repeatedly.

\begin{lemma}
 \label{F:lemma:tech_alpha_beta}
Suppose that $\lambda>(5-\sqrt{13})/2 \cong 0.697$.  Then, for any pair of consecutive elements $i<j$ in $I$ with $F(\ux_i,\ux_j)\neq 0$, we have
\[
 X_{i+1}^{2\lambda} \ll X_j\ll X_{j+1}^\alpha,
 \quad
 X_{j+1}\ll X_j^{2/\lambda} \ll X_{i+1}^\beta
\]
where $\alpha=2\lambda(1-\lambda)/(3\lambda-1)$ and $\beta=2(1-\lambda)/(3\lambda-2)$.
\end{lemma}

\begin{proof}
For such pairs $(i,j)$, Proposition \ref{F:prop:Fnotzero} gives $X_{i+1}^{2\lambda} \ll X_j$.  Combining this with the estimate $X_j \ll (X_{i+1}X_{j+1})^{1-\lambda}$ from Proposition \ref{prelim:prop:height}, we obtain $X_j \ll X_{j+1}^\alpha$.  On the other hand, by Corollary \ref{prelim:cor:est_varphi}, we have $X_j^{-2} \ll L_j \ll X_{j+1}^{-\lambda}$, and so $X_{j+1} \ll X_j^{2/\lambda}$.  Combining this with the same estimate from Proposition \ref{prelim:prop:height} yields $X_j^{2/\lambda} \ll X_{i+1}^\beta$.
\end{proof}

\begin{cor}
 \label{F:cor:upperbound_lambda}
Suppose that $F(\ux_i,\ux_j)\neq 0$ for infinitely many pairs of consecutive integers $i<j$ in $I$.  Then, we have $\lambda \le 5/7 \cong 0.714$.
\end{cor}

\begin{proof}
Assuming, as we may, that $\lambda > (5-\sqrt{13})/2$, Lemma \ref{F:lemma:tech_alpha_beta} gives $X_{i+1}^{2\lambda} \ll X_j$ and $X_j \ll X_{i+1}^{\lambda\beta/2}$ for each of these pairs $(i,j)$.  Therefore, we must have $4\le \beta$ and so $\lambda\le 5/7$.
\end{proof}

%
%

\section{Search for algebraic relations}
\label{sec:search}

\subsection{The basic search}
Given triples of indeterminates $\ux=(x_0,x_1,x_2)$ and $\uy=(y_0,y_1,y_2)$, our aim is to find non-zero polynomials $P\in\bZ[\ux,\uy]$ which vanish at the point $(\ux_i,\ux_j)\in\bZ^3\times\bZ^3$ for infinitely many pairs of consecutive elements $i<j$ in $I$.  Such vanishing should derive simply from the integral nature of the points $\ux_k$, the general growth of $\|\ux_k\|$, and the inequality
\begin{equation}
 \label{search:eqL}
  L_\xi(\ux_k) \ll \|\ux_{k+1}\|^{-\lambda},
\end{equation}
independently of the value of the implied constant.  As the latter condition remains satisfied if we replace the sequence $(\ux_k)_{k\ge 1}$ by $(a_k\ux_k)_{k\ge 1}$ for bounded non-zero integers $a_k$, we deduce that the polynomials $P(a\ux,b\uy)$ should share the same vanishing for any choice of non-zero integers $a$ and $b$.  As a consequence, that vanishing applies to each of the bi-homogeneous components of $P$ and so we may restrict our search to bi-homogeneous polynomials, namely polynomials that are separately homogeneous in $\ux$ and in $\uy$.  Similarly the condition \eqref{search:eqL} is preserved if we replace the number $\xi$ by $a\xi$ for some non-zero integer $a$ and replace each $\ux_k$ by its image under the polynomial map $\theta_a(\ux) = (x_0, ax_1, a^3x_2)$.  Thus, the polynomials $P(\theta_a(\ux),\theta_a(\uy))$ should also have the same vanishing for any non-zero integer $a$. In particular, that vanishing applies to each of the homogeneous components of $P$ for the \emph{weight}, upon defining the weight of a monomial $x_0^{e_0}x_1^{e_1}x_2^{e_2}y_0^{f_0}y_1^{f_1}y_2^{f_2}$ as $e_1+3e_2+f_1+3f_2$.  So, we may further restrict our search to weight-homogeneous polynomials.

Note that the polynomials
\begin{equation}
\label{search:eq:STUV}
 S=\varphi(\ux),\quad
 T=\Phi(\ux,\ux,\uy),\quad
 U=\Phi(\ux,\uy,\uy),\quad
 V=\varphi(\uy)
\end{equation}
obtained by polarization of $\varphi(\ux)$ are bi-homogeneous as well as homogeneous for the weight, of weight $3$ equal to their total degree.  So, any polynomial in $S,T,U,V$ which is bihomogeneous as a polynomial in $\ux$ and $\uy$ is automatically homogeneous for the weight.  This makes the subring $\bQ[S,T,U,V]$ of $\bQ[\ux,\uy]$ particularly pleasant to work with.  It is within that ring that we will search for polynomials.  Its generators $S$, $T$, $U$ and $V$ are algebraically independent over $\bQ$, as a short computation shows that their images under the specialization $x_0\mapsto 0$ and $y_0\mapsto 1$ are so.  Thus, $\bQ[S,T,U,V]$ can be viewed as a ring of polynomials in $4$ variables.

The bi-degree and the weight give rise to an $\bN^3$-grading on the ring $\bQ[\ux,\uy]$.  But since we will work in the subring $\bQ[S,T,U,V]$, it is only the $\bN^2$-grading given by the bi-degree that will matter.  So, we simply consider the $\bQ$-vector space decomposition
\[
 \bQ[\ux,\uy] = \bigoplus_{(m,n)\in\bN^2} \bQ[\ux,\uy]_{(m,n)}
\]
where $\bQ[\ux,\uy]_{(m,n)}$ stands for the bi-homogeneous part of $\bQ[\ux,\uy]$ of bi-degree $(m,n)$.


\subsection{Forcing divisibility}
For each $i\in I$, we denote by a subscript $i$ the values of the polynomials \eqref{search:eq:STUV} at the point $(\ux_i,\ux_j)$, where $j$ stands for the successor of $i$ in $I$. Thus, we have
\[
 S_i=\varphi(\ux_i),\quad
 T_i=\Phi(\ux_i,\ux_i,\ux_j),\quad
 U_i=\Phi(\ux_i,\ux_j,\ux_j),\quad
 V_i=\varphi(\ux_j).
\]
We also note that $V_i=S_j$.  Since $\ux_i$ and $\ux_{i+1}$ form a basis of the integer points in $W_i$ and since $\ux_j$ is a primitive point of $\bZ^3$, we can write
\begin{equation}
 \label{search:eqpq}
 \ux_j = p_i\ux_i+q_i\ux_{i+1}
\end{equation}
for relatively prime integers $p_i$ and $q_i$ with $q_i\neq 0$.  The next proposition gathers several properties of the integers $q_i$.

\begin{prop}
\label{search:prop:qi}
For each pair of consecutive elements $i<j$ in $I$, we have
\begin{itemize}
\item[(a)] $|q_i| \asymp X_j/X_{i+1}$,
\item[(b)] $T_i\equiv 3p_iS_i,\ U_i\equiv 3p_i^2S_i,\ V_i\equiv p_i^3S_i \mod q_i$,
\item[(c)] $\gcd(q_i,S_i)=\gcd(q_i,V_i)$.
\end{itemize}
\end{prop}

\begin{proof} Taking the exterior product of both sides of \eqref{search:eqpq} with $\ux_i$, we find $\ux_i\wedge\ux_j = q_i\ux_i\wedge\ux_{i+1}$.  Then, applying Lemma \ref{prelim:lemma:xjwedgexi} separately to each product yields $X_jL_i\asymp |q_i|X_{i+1}L_i$ and (a) follows.  The equality \eqref{search:eqpq} also yields $\ux_j\equiv p_i\ux_i \mod q_i$ and thus $T_i \equiv \Phi(\ux_i,\ux_i,p_i\ux_i) = 3p_iS_i \mod q_i$.  The other two congruences from (b) are proved in the same way.  Finally (c) is an immediate consequence of the congruence $V_i\equiv p_i^3S_i \mod q_i$ together with the fact that $p_i$ is prime to $q_i$.
\end{proof}

In particular, the above congruences imply that the integer $T_i^2-3S_iU_i$ is congruent to $0$ modulo $q_i$ and thus divisible by $q_i$.   The next observation is crucial for the present work and will allow us to reach higher divisibility properties.

\begin{lemma}
\label{search:lemma:equiv_cond}
Let $p$ and $q$ be indeterminates over the ring $\bQ[\ux,\uy]$, let $k\in\bN$ and let $P(\ux,\uy)$ be a bihomogeneous element of $\bQ[\ux,\uy]$.  Then the following assertions are equivalent
\begin{itemize}
 \item[(i)] $q^k$ divides $P(\ux,\ux+q\uy)$ in $\bQ[q,\ux,\uy]$,
 \item[(ii)] $q^k$ divides $P(\ux,p\ux+q\uy)$ in $\bQ[p,q,\ux,\uy]$.
\end{itemize}
\end{lemma}

\begin{proof}
It is clear that (ii) implies (i).  To prove the converse, suppose that $P(\ux,\ux+q\uy)=q^kQ(q,\ux,\uy)$ for some $Q\in \bQ[q,\ux,\uy]$.  Substituting $p\ux$ for $\ux$ in this equality and denoting by $d$ the degree of $P$ in $\ux$, we obtain
\[
 p^d P(\ux,p\ux+q\uy) = P(p\ux,p\ux+q\uy) = q^k Q(q,p\ux,\uy).
\]
Thus $q^k$ divides $p^d P(\ux,p\ux+q\uy)$ in $\bQ[p,q,\ux,\uy]$, and (ii) follows.
\end{proof}

\begin{definition}
\label{search:def:Jk}
For each $k\in\bN$, we denote by $J^{(k)}$ the ideal of $\bQ[\ux,\uy]$ generated by the bihomogeneous elements of $\bQ[\ux,\uy]$ which satisfy the equivalent conditions of the lemma.
\end{definition}

It can be shown that $J:=J^{(1)}$ is the ideal of $\bQ[\ux,\uy]$ generated by the coordinates of the exterior product $\ux\wedge \uy$ in the standard basis, and that $J^{(k)} = J^k$ is the $k$-th power of $J$ for each integer $k\ge 1$.  However, we will not need this fact here and this is why we adopt a different notation.  Our interest in these ideals is motivated by the following result.

\begin{prop}
\label{search:prop:divJk}
Let $k\in\bN^*$ and let $P$ be a bi-homogeneous element of $J^{(k)}\cap \bZ[\ux,\uy]$.  Then, for each pair of consecutive elements $i<j$ of $I$, the integer $P(\ux_i,\ux_j)$ is divisible by $q_i^k$.  In particular, when $P(\ux_i,\ux_j)\neq 0$, we have $|q_i|^k \le |P(\ux_i,\ux_j)|$.
\end{prop}

\begin{proof}
Since $P\in J^{(k)}$, we have $P(\ux,p\ux+q\uy) = q^k Q(p,q,\ux,\uy)$ for some polynomial $Q$ with coefficients in $\bQ$.  Since $P$ has integer coefficients, the same is true of $Q$.  The first assertion follows by specializing $p$, $q$, $\ux$ and $\uy$ in $p_i$, $q_i$, $\ux_i$ and $\ux_{i+1}$ respectively.
\end{proof}

Let $\rho$ denote the automorphism of the $\bQ$-algebra $\bQ[q,\ux,\uy]$ which fixes $q$ and $\ux$ but maps $\uy$ to $\ux+q\uy$.  According to Definition \ref{search:def:Jk}, a bi-homogeneous element $P$ of $\bQ[\ux,\uy]$ belongs to $J^{(k)}$ for some integer $k\ge 0$ if and only if $q^k$ divides $\rho(P)$ in $\bQ[q,\ux,\uy]$. When $P\neq 0$, there exists a largest integer $k$ with that property.  We call it the \emph{$J$-valuation} of $P$ and denote it $v_J(P)$.  Clearly it satisfies $v_J(PQ)=v_J(P)+v_J(Q)$ for any pair of non-zero bi-homogeneous elements $P$ and $Q$ of $\bQ[\ux,\uy]$.

A quick computation shows that
\begin{equation}
\label{search:eq:rhoSTUV}
 \begin{aligned}
  \rho(S)&=\varphi(\ux)=S,\\
  \rho(T)&=\Phi(\ux,\ux,\ux+q\uy)=3S+qT,\\
  \rho(U)&=\Phi(\ux,\ux+q\uy,\ux+q\uy)=3S+2qT+q^2U, \\
  \rho(V)&=\varphi(\ux+q\uy)=S+qT+q^2U+q^3V.
 \end{aligned}
\end{equation}
From this we deduce that the polynomials
\[
 A:=T^2-3SU \in \bQ[\ux,\uy]_{(4,2)},
 \quad
 B:=T^3-3TA-27S^2V \in \bQ[\ux,\uy]_{(6,3)}
\]
satisfy
\begin{equation}
\label{search:eq:rhoAB}
 \begin{aligned}
 \rho(A)&=(3S+qT)^2-3S(3S+2qT+q^2U)=q^2A,\\
 \rho(B)&=(3S+qT)^3-3(3S+qT)q^2A-\dots 
        =q^3B.
 \end{aligned}
\end{equation}
Therefore, they belong respectively to $J^{(2)}$ and $J^{(3)}$.  More precisely, $A$ and $B$ have respective $J$-valuations $2$ and $3$, while $S$, $T$, $U$, $V$ have valuation $0$.

According to Proposition \ref{search:prop:divJk}, the fact that $A$ belongs to $J^{(2)}$ implies that $q_i^2$ divides $A_i:=T_i^2-3S_iV_i$ for each $i\in I$, strengthening the remark made just after Proposition \ref{search:prop:qi}.  In this context, we note that
\[
 F:=F(\ux,\uy)=T^2-4SU=(4A-T^2)/3.
\]
So, if for consecutive elements $i<j$ in $I$ the integer $F_i:=F(\ux_i,\ux_j)$ vanishes, then $T_i^2=4A_i$ is divisible by $4q_i^2$ and thus $q_i$ divides $T_i$.  Taking into account the congruences of Proposition \ref{search:prop:qi} (b), we deduce that $q_i$ also divides $3p_iS_i$, $U_i$ and $3V_i$.  Since $q_i$ is relatively prime to $p_i$, this proves the first part of the following proposition.

\begin{prop}
\label{search:prop:Fzero}
For each $i\in I$ such that $F_i=0$, the integer $q_i$ divides $3S_i$, $T_i$, $U_i$ and $3V_i$.  In particular, if $\lambda>2/3$, we have $|q_i| \le 3|S_i|$ for all such large enough indices $i$.
\end{prop}

The second part follows from the fact that, when $\lambda>2/3$, the integers $S_i$ are all non-zero except for finitely many indices $i$.  In the next section, we will analyze the consequences of this result and show that, if $F_i=0$ for infinitely many $i\in I$, then $\lambda\le (5-\sqrt{13})/2\cong 0.697$.


\subsection{The ring $\cR$}
From now on, we restrict our attention to the graded ring
\[
 \cR=\bigoplus_{\ell\ge 0}\cR_\ell
  \quad \text{where} \quad
 \cR_\ell := \bQ[S,T,U,V] \cap \bQ[\ux,\uy]_{(2\ell,\ell)} \quad (\ell\ge 0),
\]
which contains $F$ but also $T$, $A$, $B$ and $S^2V$. We will say that an element $P$ of $\cR$ is \emph{homogeneous of degree $\ell$} if it belongs to $\cR_\ell$ (which is equivalent to asking that $P$ is homogeneous of degree $\ell$ in $\uy$).  Thus, $T$, $A$, $F$, $B$ and $S^2V$ are homogeneous of respective degrees 1, 2, 2, 3 and 3. The next result provides two presentations of $\cR$ as a weighted polynomial ring in three variables.

\begin{prop}
\label{search:prop:R}
We have $\cR=\bQ[T,F,S^2V]=\bQ[T,A,B]$.  For each $\ell\ge 0$, one basis of the vector space $\cR_\ell$ over $\bQ$ consists of the products $T^{\ell-2m-3n}F^m(S^2V)^n$ where $(m,n)$ runs through all pairs of non-negative integers $m$ and $n$ with $2m+3n\le \ell$.  Another basis consists of the products $T^{\ell-2m-3n}A^mB^n$ for the same pairs $(m,n)$.
\end{prop}

\begin{proof}
We first note that, for each $a,b,c,d\in \bN$, the monomial $S^aT^bU^cV^d$ is a bihomogeneous element of $\bQ[\ux,\uy]$ of bidegree $(3a+2b+c,b+2c+3d)$.  So it belongs to $\cR$ if and only if $3a+2b+c=2(b+2c+3d)$, a condition that amounts to $a=c+2d$.  Since $S$, $T$, $U$ and $V$ are algebraically independent over $\bQ$, this implies that the products
\[
 S^{c+2d}T^bU^cV^d = T^b (SU)^c (S^2V)^d \quad (b,c,d\in\bN)
\]
form a basis of $\cR$ as a vector space over $\bQ$, so $\cR=\bQ[T,SU,S^2V]$ is a polynomial ring in $3$ variables, and the first assertion of the proposition is easily verified.  The other assertions follow from this in view of the degrees of $T$, $A$, $B$, $F$ and $S^2V$.
\end{proof}

\begin{definition}
For each $\ell\in\bZ$, we denote by $\tau(\ell)$ the number of pairs $(m,n)\in\bN^2$ with $2m+3n\le \ell$.
\end{definition}

In particular, this gives $\tau(\ell)=0$ when $\ell<0$. With this notation at hand, we can now state and prove the main result of this section.

\begin{theorem}
\label{search:thm:dim}
For each choice of integers $k,\ell\ge 0$, we have
\[
 \dim_\bQ \cR_\ell = \tau(\ell)
 \et
 \dim_\bQ \frac{\cR_\ell}{\cR_\ell\cap J^{(k)}}
  = \begin{cases}
     \tau(k-1) &\text{if $k\le \ell$,}\\
     \tau(\ell) &\text{if $k > \ell$.}
    \end{cases}
\]
\end{theorem}

\begin{proof}
The first formula $\dim_\bQ \cR_\ell = \tau(\ell)$ is an immediate consequence of the previous proposition.  It tells us that, in order to prove the second one, we may restrict to $k\le \ell+1$ because, for $k=\ell+1$, the combination of the two formulas implies that $\cR_\ell\cap J^{(\ell+1)}=0$ and therefore $\cR_\ell\cap J^{(k)}=0$ whenever $k>\ell$.

For any $(m,n)\in\bN^2$ with $2m+3n\le \ell$, the formulas \eqref{search:eq:rhoSTUV} and \eqref{search:eq:rhoAB} give
\begin{equation}
 \label{search:thm:dim:eq1}
 \rho(T^{\ell-2m-3n}A^mB^n) = q^{2m+3n}(3S+qT)^{\ell-2m-3n}A^mB^n,
\end{equation}
and so $T^{\ell-2m-3n}A^mB^n$ belongs to $J^{(k)}$ if $2m+3n\ge k$.  In view of the preceding proposition, this means that the quotient $\cR_\ell/(\cR_\ell\cap J^{(k)})$ is generated, as a $\bQ$-vector space, by the classes of the products $T^{\ell-2m-3n}A^mB^n$ where $(m,n)$ runs through the elements of $\bN^2$ with $2m+3n<k$ (recall that $k\le \ell+1$).  So, it remains to prove that these classes are linearly independent over $\bQ$ and for this, we may further assume that $k\ge 1$.

Suppose on the contrary that there exist rational numbers $a_{m,n}$ not all zero, indexed by pairs $(m,n)\in\bN^2$ with $2m+3n < k$, such that
\[
 \sum_{2m+3n<k} a_{m,n}T^{\ell-2m-3n}A^mB^n \in J^{(k)}.
\]
Let $r=\min\{2m+3n\,;\, a_{m,n}\neq 0\}$.  By the above, we obtain that
\[
 P:=\sum_{2m+3n=r} a_{m,n}T^{\ell-r}A^mB^n \in J^{(r+1)}.
\]
However, the formula \eqref{search:thm:dim:eq1} implies that
\begin{align*}
 \rho(P)
 &= \sum_{2m+3n=r} a_{m,n}q^r(3S+qT)^{\ell-r}A^mB^n \\
 &\equiv q^r(3S)^{\ell-r} \sum_{2m+3n=r} a_{m,n}A^mB^n \mod q^{r+1}.
\end{align*}
Thus $\rho(P)$ is not divisible by $q^{r+1}$, and so $P\notin J^{(r+1)}$, a contradiction.
\end{proof}

\subsection{Examples}
Besides $S=\varphi$ and $F$, the proof of Theorem \ref{intro:thm:main} uses three more auxiliary polynomials $D^{(2)}$, $D^{(3)}$ and $D^{(6)}$ that we now introduce in the form of examples that illustrate the above considerations.  Their superscript refers to their $J$-valuation.

\medskip
\begin{example}
Since $\dim_\bQ(\cR_6/(\cR_6\cap J^{(2)})) = \tau(1)=1$, any pair of elements of $\cR_6$ are $\bQ$-linearly dependent modulo $J^{(2)}$.  Since
\[
 F = \frac{1}{3}(4A-T^2) \equiv -\frac{T^2}{3}
 \et
 S^2V = \frac{1}{27}(T^3-3TA-B) \equiv \frac{T^3}{27} \mod J^{(2)},
\]
we obtain
\begin{equation*}
\label{search:eq:D2}
 D^{(2)} := F^3+27(S^2V)^2 \in \cR_6\cap J^{(2)}.
\end{equation*}
\end{example}

\medskip
\begin{example}
Since $\dim_\bQ(\cR_6/(\cR_6\cap J^{(3)})) = \tau(2)=2$, the products $F^3$, $TF(S^2V)$ and $(S^2V)^2$ are $\bQ$-linearly dependent modulo $J^{(3)}$.  We find that
\[
 F^3 \equiv \frac{12T^4A-T^6}{27},
 \quad
 TF(S^2V) \equiv \frac{7T^4A-T^6}{81},
 \quad
 (S^2V)^2 \equiv \frac{T^6-6T^4A}{27^2}
\]
modulo $J^{(3)}$, and thus
\begin{equation*}
\label{search:eq:D3}
 D^{(3)} := F^3-18TF(S^2V)-135(S^2V)^2 \in \cR_6\cap J^{(3)}.
\end{equation*}
\end{example}

\medskip
\begin{example}
 \label{search:example:D6}
Since $\dim_\bQ(\cR_9/(\cR_9\cap J^{(6)})) = \tau(5)=5$, the six products $TF^4$, $F^3(S^2V)$, $T^2F^2(S^2V)$, $TF(S^2V)^2$, $T^3(S^2V)^2$ and $(S^2V)^3$ are $\bQ$-linearly dependent modulo $J^{(6)}$.  Explicitly, this yields
\begin{equation*}
\label{search:eq:D6}
 \begin{aligned}
  D^{(6)} := TF^4&+10F^3(S^2V)-11T^2F^2(S^2V)\\
                 &-180TF(S^2V)^2-T^3(S^2V)^2-675(S^2V)^3 \in \cR_9\cap J^{(6)}.
 \end{aligned}
\end{equation*}
Instead of checking this relation by a direct computation, it is simpler and more useful to derive it from an alternative formula for $D^{(6)}$.  To this end, we define
\begin{equation}
\label{search:eq:MN}
 M:=F^2-3TS^2V,
 \quad
 N:=D^{(3)}=F^3-18TFS^2V-135(S^2V)^2.
\end{equation}
It is easy to verify that $M\in\cR_4\cap J^{(2)}$, and we already know that $N\in\cR_6\cap J^{(3)}$.  Thus any $\bQ$-linear combination of $M^3$ and $N^2$ belongs to $\cR_{12}\cap J^{(6)}$.  On the other hand, since
\[
 M\equiv F^2 \et N\equiv F^3 \mod S^2V,
\]
we find that $S^2V$ divides $M^3-N^2$.  Since $S^2V$ has $J$-valuation $0$, the quotient is an element of $\cR_9\cap J^{(6)}$.  Expanding the expression $M^3-N^2$, we find that this quotient is $27D^{(6)}$, namely
\begin{equation}
\label{search:eq:D6bis}
 M^3-N^2 = 27S^2V D^{(6)}.
\end{equation}
\end{example}

We will need the following consequence of these constructions.

\begin{prop}
\label{search:prop:D2D3D6}
Suppose that $\lambda > (5-\sqrt{13})/2\cong 0.697$. For each $i\in I$, we have
\begin{itemize}
\item[(a)] $|q_i|^2 \ll |F_i|^3 + |S_i^2V_i|^2$ \quad\quad\ \ if \ $D^{(2)}_i\neq 0$,
\item[(b)] $|q_i|^3 \ll |F_i|^3 + |T_iF_iS_i^2V_i|$ \quad\    if \ $F_i\neq 0$ and $D^{(3)}_i\neq 0$,
\item[(c)] $|q_i|^6 \ll |T_iF_i^4| + |T_i^3(S_i^2V_i)^2|$ \   if \ $F_i\neq 0$ and $D^{(6)}_i\neq 0$.
\end{itemize}
\end{prop}

\begin{proof}
The statement (a) is a direct consequence of Proposition \ref{search:prop:divJk} applied to the polynomial $D^{(2)}$, and does not require any hypothesis on $\lambda$.  To prove (b) and (c), we first note that, by Propositions \ref{prelim:prop:est_Phi} and \ref{F:prop:Fnotzero} we have $|T_i|\asymp (X_j/X_i)|S_i|$ and $|S_iV_i| = o(X_j/X_i)$ for $i<j$ running through all pairs of consecutive elements of $I$ with $i\to\infty$.  Therefore
\[
 |S_i^2V_i| = o(|T_i|)=o(|T_iF_i|)
\]
as $i$ goes to infinity through elements of $I$ with $F_i\neq 0$ (assuming, as we may, that there are infinitely many such $i$).  Since all monomials that compose $D^{(6)}$ can be obtained by multiplying $TF^4$, $T^2F^2(S^2V)$ and $T^3(S^2V)^2$ by appropriate powers of $S^2V/(TF)$, we deduce that, for the same values of $i$, we have
\[
 |D_i^{(6)}| \ll |T_iF_i^4| + |T_i^2F_i^2S_i^2V_i| + |T_i^3(S_i^2V_i)^2| \ll |T_iF_i^4| + |T_i^3(S_i^2V_i)^2|
\]
(in the second estimate, we dropped the middle term as it is the geometric mean of the other two).  Similarly, we find that $|D_i^{(3)}| \ll |F_i|^3 + |T_iF_iS_i^2V_i|$.  Then, (b) and (c) follow from Proposition \ref{search:prop:divJk}.
\end{proof}


\subsection{An additional remark}
\label{sec:search:remark}

Since $\cR$ is a polynomial ring over $\bQ$ in the variables $T$, $F$ and $S^2V$, it is a unique factorization domain and it makes sense to talk about irreducible elements of $\cR$ although, a priori, such polynomials may not remain irreducible in the ring $\bQ[\ux,\uy]$.  Similarly we can talk about relatively prime elements of $\cR$.  One can show that the polynomials $D^{(2)}$, $D^{(3)}$ and $D^{(6)}$ constructed above are irreducible elements of $\cR$.  Clearly, $F$ is another one. Therefore the next proposition implies that, for each sufficiently large $i\in I$, at most one of the integers $F_i$, $D^{(2)}_i$, $D^{(3)}_i$ or $D^{(6)}_i$ is zero.

\begin{prop}
 \label{search:prop:remark}
Suppose that $\lambda>2/3$, and let $P$, $Q$ be relatively prime homogeneous elements of $\cR=\bQ[T,F,S^2V]$.  Then there are only finitely many $i\in I$ such that $P$ and $Q$ both vanish at the point $(T_i,F_i,S_i^2V_i)$.
\end{prop}

The proof uses the following estimate that we will also need later for other purposes.

\begin{lemma}
 \label{search:lemma:oTcube}
Suppose that $\lambda>2/3$.  Then, $|S_i^2V_i| = o(|T_i|^3)$ for $i\in I$.
\end{lemma}

\begin{proof}
For consecutive elements $i<j$ in $I$ with $i$ large enough so that $S_i$, $T_i$ and $V_i$ are non-zero, Proposition \ref{prelim:prop:est_Phi} gives $|T_i| \asymp (X_j/X_i)|S_i|$.  Then, as $|S_i|\ge 1$, this yields $|S_i^2V_i/T_i^3| \asymp (X_i/X_j)^3 |V_i/S_i| \le (X_i/X_j)^3 |V_i|$.  By Corollary \ref{prelim:cor:est_varphi}, we also have $|V_i|\ll X_j^2X_{j+1}^{-\lambda} \le X_j^{2-\lambda} \le X_j^{4/3}$, while Corollary \ref{prelim:cor:XjLj} gives $X_i^2 \ll X_{i+1}\le X_j$.  Combining these estimates, we conclude that $|S_i^2V_i/T_i^3| \ll (X_i/X_j)^3 X_j^{4/3} \ll X_j^{-1/6}=o(1)$.
\end{proof}

\begin{proof}[Proof of Proposition \ref{search:prop:remark}]
The hypothesis implies that the de-homogenized polynomials $\bar{P}=P(1,F/T^2,S^2V/T^3)$ and $\bar{Q}=Q(1,F/T^2,S^2V/T^3)$ are relatively prime elements of $\bQ[F/T^2,S^2V/T^3]$ viewed as polynomials in $2$ variables.  Therefore, as such, they have at most finitely many common zeros in $\Qbar^2$.

Since $\lambda>2/3$, it follows from Proposition \ref{prelim:prop:est_Phi} that there exists an index $i_0$ such that $S_i$, $T_i$ and $V_i$ are all non-zero for each $i\in I$ with $i\ge i_0$.  For those $i$, the ratio $S_i^2V_i/T_i^3$ is a non-zero rational number and, by the previous lemma, it tends to $0$ as $i\to\infty$.  Thus, there are only finitely many values of $i\ge i_0$ for which $(F_i/T_i^2,S_i^2V_i/T_i^3)$ is a common zero of $\bar{P}$ and $\bar{Q}$, and so there are only finitely many $i\in I$ such that $P$ and $Q$ vanish at $(T_i,F_i,S_i^2V_i)$.
\end{proof}

%
%

\section{Non-vanishing of F}
\label{sec:NVF}

The main goal of this section is to show that $F_i\neq 0$ for any sufficiently large $i\in I$ if $\lambda >(5-\sqrt{13})/2$.  The following result is a first step.

\begin{prop}
\label{NVF:prop:allF0}
Suppose that $F_i=0$ for all but finitely many $i\in I$.  Then $\lambda\le 2/3$.
\end{prop}

\begin{proof}
We proceed by contradiction assuming, on the contrary, that $\lambda>2/3$.  Then, according to Theorem \ref{prelim:thm:phi}, we have $S_i\neq 0$ for all but finitely many indices $i$ and so there exists an integer $i_0$ such that $F_i=0$ and $S_i\neq 0$ for all $i\ge i_0$.  Put $\epsilon=\lambda-2/3$.  For each pair of consecutive elements $i<j$ in $I$ with $i\ge i_0$, Corollary \ref{F:cor:Fzero} together with the estimate $|q_i|\asymp X_j/X_{i+1}$ of Proposition \ref{search:prop:qi} yields
\[
 |S_iS_j| \ll X_i^{-1} X_{i+1}^{-8/3-4\epsilon} X_j^3
          \asymp |q_i|^3 (X_{i+1}^{1/3}/X_i) X_{i+1}^{-4\epsilon}.
\]
As $|S_i| \ll X_i^2X_{i+1}^{-\lambda} \ll (X_i/X_{i+1}^{1/3})^2$ (by Corollary \ref{prelim:cor:est_varphi}), we deduce that
\[
 |S_iS_j| \ll |q_i|^3 |S_i|^{-1/2}X_{i+1}^{-4\epsilon}.
\]
By Proposition \ref{search:prop:Fzero}, we also have $|q_i|\le 3|S_i|$ and $|q_j|\le 3|S_j|$.  Thus the above estimate yields
\begin{equation}
\label{NVF:prop:allF0:eq1}
 |q_j| \ll |q_i|^{3/2}X_{i+1}^{-4\epsilon}.
\end{equation}
In particular, for $i$ large enough, we have
\[
 \log |q_j| \le \frac{3}{2} \log |q_i|.
\]
On the other hand, since $\lambda\ge 2/3$, we also have $X_i^2 \ll X_{i+1} \le X_j$ by Corollary \ref{prelim:cor:XjLj}, thus $\log X_j \ge (2+o(1))\log X_i$ as $i$ goes to infinity in $I$, and so
\[
 \frac{\log |q_j|}{\log X_j} \le \left(\frac{3}{4}+o(1)\right) \frac{\log |q_i|}{\log X_i}
\]
showing that the ratio $\log |q_i|/\log X_i$ tends to $0$ as $i$ goes to infinity in $I$.  In particular, we must have $1\le |q_i|\le X_i^\epsilon$ for each sufficiently large $i\in I$, in contradiction with \eqref{NVF:prop:allF0:eq1}.
\end{proof}

\begin{theorem}
\label{NVF:thm}
Suppose that $\lambda>(5-\sqrt{13})/2 \cong 0.697$.  Then, we have $F_i\neq 0$ for each sufficiently large element $i$ of $I$.
\end{theorem}

\begin{proof}
Assume, on the contrary, that $F_i=0$ for infinitely many values of $i$.  Then, by the previous proposition, there exist arbitrarily large elements $i$ of $I$ such that, upon denoting by $j$ the next element of $I$, we have $F_i\neq0$ and $F_j=0$.  We may further assume, by Theorem \ref{prelim:thm:phi}, that $S_j\neq 0$ and $S_k\neq 0$ where $k$ is the next element of $I$ after $j$.  Since $F_j=0$, Corollary \ref{F:cor:Fzero} gives
\begin{equation*}
\label{NVF:thm:eq1}
 |S_j| \le |S_jS_k| \ll X_j^{-1}X_{j+1}^{-4\lambda}X_k^3.
\end{equation*}
By Propositions \ref{search:prop:qi} and \ref{search:prop:Fzero}, we also have $X_k \asymp |q_j| X_{j+1} \ll |S_j| X_{j+1}$.  Substituting this upper bound for $X_k$ in the previous estimate and then using the standard upper bound $|S_j|\ll X_j^2X_{j+1}^{-\lambda}$ from Corollary \ref{prelim:cor:est_varphi}, we obtain
\[
 1 \ll |S_j|^2 X_j^{-1} X_{j+1}^{3-4\lambda} \ll X_j^3 X_{j+1}^{3-6\lambda},
\]
and so $X_j\gg X_{j+1}^{2\lambda-1}$.  On the other hand, since $F_i\neq 0$, Lemma \ref{F:lemma:tech_alpha_beta} yields $X_j\ll X_{j+1}^\alpha$ where $\alpha=2\lambda(1-\lambda)/(3\lambda-1)$.  As $j$ can be made arbitrarily large, we deduce that $2\lambda -1 \le \alpha$ and so $\lambda \le (7+\sqrt{17})/16 \cong 0.695$, a contradiction.
\end{proof}

\begin{cor}
\label{NVF:cor:oT}
Suppose that $\lambda>(5-\sqrt{13})/2 \cong 0.697$.  Then, we have $|F_iS_i^2V_i| = o(|T_i|)$ as $i$ goes to infinity in $I$.
\end{cor}

\begin{proof}
According to Proposition \ref{prelim:prop:est_Phi}, we have $T_i\neq 0$ and $|T_i|\asymp (X_j/X_i) |S_i|$ for each large enough $i\in I$.  By the theorem, we also have $F_i\neq 0$.  Then, upon denoting by $j$ the successor of $i$ in $I$, Proposition \ref{F:prop:Fnotzero} gives $|F_i|\ll X_j^2L_i^4$.  Combining this with the estimates of Corollary \ref{prelim:cor:est_varphi} for $|S_i|$ and $|V_i|=|S_j|$, we obtain
\[
 \left|\frac{F_iS_i^2V_i}{T_i}\right|
 \asymp \frac{X_i}{X_j}|F_iS_iV_i|
 \ll \frac{X_i}{X_j}(X_j^2L_i^4)(X_i^2L_i)(X_j^2L_j)
 \ll X_i^3X_{i+1}^{-5\lambda}X_j^3X_{j+1}^{-\lambda}.
\]
Assuming $i$ large enough, we also have $F_h\neq 0$ where $h$ stands for the predecessor of $i$ in $I$.  Then Lemma \ref{F:lemma:tech_alpha_beta} gives $X_i\ll X_{i+1}^\alpha$ (since $F_h\neq 0$) and $X_{j+1}\ll X_{i+1}^\beta$ (since $F_i\neq 0$).  We first use the estimate $X_j \ll (X_{i+1}X_{j+1})^{1-\lambda}$ from Proposition \ref{prelim:prop:height} to eliminate $X_j$ from our upper bound for $|F_iS_i^2V_i/T_i|$, and then we use the preceding estimates to eliminate $X_i$ and $X_{j+1}$. This yields $|F_iS_i^2V_i/T_i| \ll X_{i+1}^\tau$ where
\[
 \tau = 3\alpha+(3-8\lambda)+(3-4\lambda)\beta
      = \frac{\lambda(2\lambda-1)(23-33\lambda)}{(3\lambda-1)(3\lambda-2)}
\]
is negative since $\lambda >23/33=0.\overline{69}$.
\end{proof}

\begin{cor}
\label{NVF:cor:oq}
Suppose that $\lambda\ge 0.6985$.  Then, $|S_i^2V_i| = o(|q_i|)$ for $i\in I$.
\end{cor}

\begin{proof}
The proof is similar to that of Corollary \ref{NVF:cor:oT}.  Using the estimate $|q_i| \asymp X_j/X_{i+1}$ from Proposition \ref{search:prop:qi}, we find
\[
 \left|\frac{S_i^2V_i}{q_i}\right|
 \ll \frac{X_{i+1}}{X_j}(X_i^2L_i)^2(X_j^2L_j)
 \ll X_i^4X_{i+1}^{1-2\lambda}X_jX_{j+1}^{-\lambda}
 \ll X_i^4X_{i+1}^{2-3\lambda}X_{j+1}^{1-2\lambda},
\]
and so $|S_i^2V_i/q_i| \ll X_{i+1}^\tau$ with
\[
 \tau = 4\alpha+(2-3\lambda)+(1-2\lambda)\frac{2\lambda}{\alpha}
      = -\frac{\lambda^3+13\lambda^2-11\lambda+1}{(3\lambda-1)(1-\lambda)}
      < 0.
\]
The difference with the proof of Corollary \ref{NVF:cor:oT} is that we used the lower bound $X_{j+1} \gg X_{i+1}^{2\lambda/\alpha}$ of Lemma \ref{F:lemma:tech_alpha_beta} to eliminate $X_{j+1}$, as it appears with the negative exponent $1-2\lambda$.
\end{proof}

The last result below motivates the non-vanishing results for $D^{(2)}$ that we prove in the next section.

\begin{prop}
\label{NVF:prop:D2nonzero}
Suppose that $\lambda\ge 0.6985$.  For any pair of consecutive elements $i<j$ of $I$ satisfying $F_i^3+27(S_i^2V_i)^2\neq 0$, with $i$ large enough so that $F_i\neq 0$, we have
\[
 |q_i|^2 \ll |F_i|^3
 \et
 X_{i+1} \ll X_j^{2/(6\lambda-1)} \ll X_{j+1}^{(2-2\lambda)/(8\lambda-3)}.
\]
If there are infinitely many such pairs $(i,j)$, then $\lambda < 0.709$.
\end{prop}

\begin{proof}  The inequality $|q_i|^2 \ll |F_i|^3$ follows immediately from Proposition \ref{search:prop:D2D3D6} (a) combined with the preceding Corollary.  Using $|q_i| \asymp X_j/X_{i+1}$ (Proposition \ref{search:prop:qi}) and $|F_i| \ll X_j^2L_i^4$ (Proposition \ref{F:prop:Fnotzero}), it yields $X_{i+1}\ll X_j^{2/(6\lambda-1)}$.  Then, using $X_j\ll (X_{i+1}X_{j+1})^{1-\lambda}$ (Proposition \ref{prelim:prop:height}) to further eliminate $X_{i+1}$, we find $X_j^{2/(6\lambda-1)} \ll X_{j+1}^{(2-2\lambda)/(8\lambda-3)}$.  This proves the first assertion of the proposition.  Moreover, the combination of $X_{i+1}\ll X_j^{2/(6\lambda-1)}$ with $X_j \ll X_{i+1}^{\lambda(1-\lambda)/(3\lambda-2)}$ (Lemma \ref{F:lemma:tech_alpha_beta}) yields $(6\lambda-1)(3\lambda-2) \le 2\lambda(1-\lambda)$ if there are infinitely many such pairs $(i,j)$, and the second assertion follows.
\end{proof}

%
%

\section{Non-vanishing of $D^{(2)}$}
\label{sec:NVD2}

The main result of this section is that $D_i^{(2)}\neq 0$ for each sufficiently large $i\in I$ if $\lambda \ge 0.6985$.  We start by establishing algebraic consequences of a possible vanishing.

\begin{prop}
\label{NVD2:prop}
Suppose that $D^{(2)}_i=F_i^3+27S_i^4V_i^2=0$ for some $i\in I$.  Then, there exists a unique integer $R_i$ for which $F_i=-3R_i^2$ and $S_i^2V_i=R_i^3$.  This integer satisfies
\begin{equation}
\label{NVD2:prop:eq1}
 \gcd(q_i,R_i)=\gcd(q_i,S_i).
\end{equation}
Moreover, $q_i^6$ divides $cS_i^7(T_i-3R_i)^2$ for some integer constant $c>0$ not depending on $i$.
\end{prop}

\begin{proof}
We rewrite the hypothesis in the form $(-F_i/3)^3=(S_i^2V_i)^2$.  Since $S_i^2V_i\in\bZ$, the first assertion of the proposition follows.  Then, using the equality $\gcd(q_i,S_i)=\gcd(q_i,V_i)$ from Proposition \ref{search:prop:qi} (c), we deduce that
\[
 \gcd(q_i,R_i)^3 = \gcd(q_i^3,S_i^2V_i) = \gcd(q_i,S_i)^3
\]
which yields \eqref{NVD2:prop:eq1}.  We also find that
\[
 \begin{aligned}
 4A_i &= T_i^2+3F_i = (T_i-3R_i)(T_i+3R_i),\\
 4B_i &= T_i^3-9T_iF_i-108S_i^2V_i = (T_i-3R_i)(T_i^2+3R_iT_i+36R_i^2).
 \end{aligned}
\]
Since $q_i^2\mid A_i$ and $q_i^3\mid B_i$, this means that $q_i^6$ divides both
\[
 (T_i-3R_i)^3(T_i+3R_i)^3 \et (T_i-3R_i)^2(T_i^2+3R_iT_i+36R_i^2)^2
\]
and therefore $q_i^6$ divides $cR_i^7(T_i-3R_i)^2$ where $c$ denotes the resultant of the polynomials $(x-3y)(x+3y)^3$ and $(x^2+3xy+36y^2)^2$.  This integer $c$ is non-zero since these polynomials have no common zero in $\bP^1(\bC)$.  Using \eqref{NVD2:prop:eq1}, we conclude that $q_i^6$ divides
\[
 c\gcd(q_i^7,R_i^7)(T_i-3R_i)^2 = c\gcd(q_i^7,S_i^7)(T_i-3R_i)^2
\]
and so it divides $cS_i^7(T_i-3R_i)^2$.
\end{proof}

\begin{xrem}
By a similar method, one can show that
\[
 q_i \mid 3S_i(T_i-3R_i),
 \quad
 q_i^2 \mid 9S_i^2(T_i-3R_i)
 \et
 q_i^3 \mid 162S_i^4(T_i-3R_i),
\]
but these divisibility relations can also be derived, up to the value of the constant, from the one of the proposition.
\end{xrem}

\begin{cor}
\label{NVD2:prop:cor}
Suppose that $\lambda>(5-\sqrt{13})/2 \cong 0.697$.  Then, for all pairs of consecutive elements $i<j$ of $I$ with $F_i^3+27S_i^4V_i^2=0$ and $i$ large enough so that $F_i\neq 0$, we have
\[
 X_i^2X_j^4 \ll X_{i+1}^6 |S_i|^9
 \et
 |S_i^2S_j| \ll X_{i+1}^{-6\lambda}X_j^3.
\]
\end{cor}

\begin{proof}
For these pairs $(i,j)$, the integer $R_i$ defined in the proposition satisfies $R_i^2\asymp |F_i|$.  Since Corollary  \ref{NVF:cor:oT} gives $|F_i|=o(|T_i|)=o(T_i^2)$, we deduce that $|T_i-3R_i| \asymp |T_i| \asymp (X_j/X_i) |S_i|$, where the last estimate comes from Proposition \ref{prelim:prop:est_Phi}.  In particular, if $i$ is large enough, the integer $cS_i^7(T_i-3R_i)^2$ is non-zero and, as it is divisible by $q_i^6$, we obtain
\[
 q_i^6 \le c|S_i|^7 |T_i-3R_i|^2 \asymp (X_j/X_i)^2 |S_i|^9.
\]
Since $|q_i| \asymp X_j/X_{i+1}$ (by Proposition \ref{search:prop:qi}), this yields the first estimate.  The second one follows directly from the upper bound $|F_i| \ll X_j^2L_i^4$ of Proposition \ref{F:prop:Fnotzero} together with $S_i^4S_j^2 = |F_i|^3/27$.
\end{proof}

\begin{theorem}
\label{NVD2:thm}
Suppose that $\lambda\ge 0.6985$.  Then, we have $F_i^3+27S_i^4V_i^2\neq 0$ for any large enough $i\in I$.
\end{theorem}

\begin{proof} Suppose on the contrary that the set $I_2:=\{i\in I\,;\, F_i^3+27S_i^4V_i^2=0\}$ is infinite.  As a first step, suppose also that $I\setminus I_2$ is infinite.  Then, there exists infinitely many triples of consecutive elements $i<j<k$ of $I$ with $i\in I\setminus I_2$ and $j\in I_2$.  Choosing them large enough, we may further assume that $F_j\neq 0$.  Since $F_j^3+27S_j^4V_j^2=0$, we have $S_j\neq 0$, $S_k=V_j\neq 0$ and the previous corollary gives
\begin{equation}
\label{NVD2:thm:eq1}
 X_j^2X_k^4 \ll X_{j+1}^6 |S_j|^9
 \et
 |S_j^2S_k| \ll X_{j+1}^{-6\lambda}X_k^3,
\end{equation}
while Corollary \ref{prelim:cor:est_varphi} gives
\begin{equation}
\label{NVD2:thm:eq2}
 |S_j| \ll X_j^2X_{j+1}^{-\lambda}.
\end{equation}
Eliminating $|S_j|$ between \eqref{NVD2:thm:eq1} and \eqref{NVD2:thm:eq2}, and using the lower bound $|S_k|\ge 1$ to further eliminate $|S_k|$ from the resulting estimates, we obtain
\[
 X_{j+1}^{9\lambda-6}X_k^4 \ll X_j^{16}
 \et
 X_j^4 X_{j+1}^{54\lambda-12} \ll X_k^{19}.
\]
Then, eliminating $X_k$, we find $X_{j+1}^{43\lambda-18} \ll X_j^{32}$. On the other hand, since $F_i^3+27S_i^4V_i^2\neq 0$, Proposition \ref{NVF:prop:D2nonzero} gives $X_j \ll X_{j+1}^{(6\lambda-1)(1-\lambda)/(8\lambda-3)}$.  Combining the latter two estimates and noting that they hold for infinitely many $j$, we conclude that $(43\lambda-18)(8\lambda-3) \le 32(6\lambda-1)(1-\lambda)$, in contradiction with our hypothesis that $\lambda\ge 0.6985$.

Thus $I\setminus I_2$ is a finite set and so, there exists an integer $i_0$ such that $F_i\neq 0$ and $F_i^3+27S_i^4V_i^2 = 0$ for each $i\in I$ with $i\ge i_0$.  Suppose now that the sequence $(|S_i|)_{i\in I}$ is unbounded. Then, there exist infinitely many triples of consecutive elements $i<j<k$ of $I$ with $i\ge i_0$ and $|S_j|\le |S_k|$.  For these triples, the inequalities \eqref{NVD2:thm:eq1} and \eqref{NVD2:thm:eq2} are again satisfied.  We use the hypothesis $|S_j|\le |S_k|$ to eliminate $|S_k|$ from the second inequality of \eqref{NVD2:thm:eq1} and then eliminate $|S_j|$ from the resulting three inequalities.  This yields
\[
 X_{j+1}^{9\lambda-6}X_k^4 \ll X_j^{16}
 \et
 X_j^2X_{j+1}^{18\lambda-6} \ll X_k^5.
\]
Then, eliminating $X_k$, we find $X_{j+1}^{13\lambda-6} \ll X_j^{8}$.  Since Lemma \ref{F:lemma:tech_alpha_beta} also gives $X_j\ll X_{j+1}^{\alpha}$ and since $j$ can be taken arbitrarily large, we conclude that $13\lambda-6 \le 8\alpha$ which again contradicts our hypothesis on $\lambda$.    This means that the sequence $(|S_i|)_{i\in I}$ is bounded and so, the first inequality of \eqref{NVD2:thm:eq1} yields $X_j^2X_k^4 \ll X_{j+1}^6$ for infinitely many pairs of consecutive elements $i<j$ of $I$.  Then using the estimates $X_{j+1}^{\lambda/2}\ll X_j$ and $X_{j+1}^{2\lambda} \ll X_k$ coming from Lemma \ref{F:lemma:tech_alpha_beta}, we conclude that $X_{j+1}^{9\lambda} \ll X_{j+1}^6$ for the same pairs $(i,j)$ and so $\lambda\le 2/3$, a contradiction.
\end{proof}

\begin{cor}
\label{NVD2:thm:cor}
Suppose that $\lambda\ge 0.7034$.  Then, $|T_iF_iS_i^2V_i|=o(|q_i|^3)$ for $i\in I$.
\end{cor}

\begin{proof}
Let $h<i<j$ be consecutive elements of $I$. Applying the usual estimates for $|q_i|$, $|T_i|$, $|F_i|$, $|S_i|$ and $|V_i|=|S_j|$, as in the proofs of Corollaries \ref{NVF:cor:oT} and \ref{NVF:cor:oq}, we find
\[
 |q_i|^{-3}|T_iF_iS_i^2V_i|
  \asymp X_i^{-1}X_{i+1}^3X_j^{-2} |F_iS_i^3V_i|
   \ll X_i^5 X_{i+1}^{3-7\lambda} X_j^2 X_{j+1}^{-\lambda}.
\]
Using the upper bounds $X_i\ll (X_{h+1}X_{i+1})^{1-\lambda}$ and $X_j \ll (X_{i+1}X_{j+1})^{1-\lambda}$ of Proposition \ref{prelim:prop:height} to eliminate $X_i$ and $X_j$, we obtain \begin{equation*}
 |q_i|^{-3}|T_iF_iS_i^2V_i|
   \ll X_{h+1}^{5-5\lambda} X_{i+1}^{10-14\lambda} X_{j+1}^{2-3\lambda}.
\end{equation*}
By Proposition \ref{NVF:prop:D2nonzero} and Theorem \ref{NVD2:thm}, we also have
\[
 X_{h+1} \ll X_{i+1}^{(2-2\lambda)/(8\lambda-3)}
 \et
 X_{i+1}^{(8\lambda-3)/(2-2\lambda)} \ll X_{j+1}.
\]
As $5-5\lambda>0>2-3\lambda$, we conclude that $|q_i|^{-3}|T_iF_iS_i^2V_i| \ll X_{i+1}^\tau$ with
\[
 \tau = (5-5\lambda)\frac{2-2\lambda}{8\lambda-3} + (10-14\lambda) + (2-3\lambda)\frac{8\lambda-3}{2-2\lambda}
      < 0.
\qedhere
\]
\end{proof}

We conclude this section with the following consequence of the above corollary.

\begin{prop}
\label{NVD2:prop:D3D6}
Suppose that $\lambda\ge 0.7034$.  Then, for each $i\in I$ large enough so that $T_i\neq 0$ and $F_i\neq 0$, we have
\[
 \begin{aligned}
  |q_i| &\ll |F_i| &&\text{if $D^{(3)}_i\neq 0$,}\\
  |q_i|^6 &\ll |T_iF_i^4| &&\text{if $D^{(3)}_i\neq 0$ and $D^{(6)}_i\neq 0$.}
 \end{aligned}
\]
\end{prop}

\begin{proof}
Suppose that $D_i^{(3)}\neq 0$.  Then, the estimate $|T_iF_iS_i^2V_i|=o(|q_i|^3)$ of the previous corollary combined with Proposition \ref{search:prop:D2D3D6} (b) yields $|q_i| \ll |F_i|$.  From this we deduce that $|T_iF_iS_i^2V_i| = o(|F_i|^3)$, so $|T_iS_i^2V_i|=o(|F_i|^2)$ and thus $|T_i^3(S_i^2V_i)^2| = o(|T_iF_i^4|)$. Combining the latter estimate with Proposition \ref{search:prop:D2D3D6} (c) yields $|q_i|^6 \ll |T_iF_i^4|$ if $D^{(6)}_i\neq 0$.
\end{proof}

%
%

\section{Non-vanishing of $D^{(3)}$ and $D^{(6)}$}
\label{sec:NVD3D6}

We first prove non-vanishing results for $D^{(3)}$ and $D^{(6)}$, and then prove Theorem \ref{intro:thm:main} as a consequence of these and of the above Proposition \ref{NVD2:prop:D3D6}.

\begin{prop}
\label{NVD3D6:propD3}
Suppose that $\lambda>(5-\sqrt{13})/2 \cong 0.697$.  Then $D_i^{(3)}\neq 0$ and $D_i^{(6)}\neq 0$ for each sufficiently large $i\in I$.
\end{prop}

\begin{proof}
Suppose that $D_i^{(3)}=0$ for some $i\in I$ large enough so that $S_i\neq 0$ and $V_i\neq 0$.  Then, we have $F_i(F_i^2-18T_iS_i^2V_i)=135S_i^4V_i^2$, so $F_i$ is a non-zero divisor of $135S_i^4V_i^2$ and therefore $1\le |F_i| \ll S_i^4V_i^2$.  On the other hand, Corollary \ref{NVF:cor:oT} shows that $|F_iS_i^2V_i|=o(|T_i|)$.  This yields
\[
 \max\{|F_i|^3,\,|S_i^4V_i^2|\}  \ll |F_i^2S_i^4V_i^2| = o(|T_iF_iS_i^2V_i|)
\]
and therefore $0 = D_i^{(3)} = -(18+o(1))T_iF_iS_i^2V_i$.  Thus $i$ must be bounded from above.  This proves the non-vanishing assertion for $D^{(3)}$.

For $D^{(6)}$, we proceed by contradiction assuming, on the contrary, that $I_6:=\{i\in I\,;\, D_i^{(6)}=0\}$ is an infinite set.  For each large enough $i\in I_6$, the integers $F_i$, $S_i$, $T_i$ and $V_i$ are all non-zero and the formula \eqref{search:eq:D6bis} of Example \ref{search:example:D6} yields $0=M_i^3-N_i^2$ where
\[
 M_i=F_i^2-3T_iS_i^2V_i
 \et
 N_i=D_i^{(3)}=F_i^3-18T_iF_iS_i^2V_i-135S_i^4V_i^2.
\]
Then, we can write $M_i=R_i^2$ and $N_i=R_i^3$ for a unique integer $R_i$ and the above formulas become
\begin{equation}
 \label{NVD3D6:propD6:eq1}
 \begin{aligned}
 F_i^2-R_i^2 &=3T_iS_i^2V_i, \\
 F_i^3-R_i^3 &=18T_iF_iS_i^2V_i+135S_i^4V_i^2.
 \end{aligned}
\end{equation}
In particular, we have $F_i-R_i\neq 0$.  Since Corollary \ref{NVF:cor:oT} shows that $S_i^2V_i=o(|T_i/F_i|)=o(|T_iF_i|)$, we deduce that
\[
 F_i^3-R_i^3 = (18+o(1))T_iF_iS_i^2V_i = (6+o(1))F_i(F_i^2-R_i^2)
\]
which, after division by $F_i^2(F_i-R_i)$, yields
\[
 (R_i/F_i)^2-(5+o(1))(R_i/F_i)-(5+o(1))=0.
\]
Thus the sequence of ratios $(R_i/F_i)_{i\in I_6}$ has at most two irrational accumulation points, and so $|R_i|\asymp |F_i| \asymp |F_i-R_i| \asymp |F_i+R_i|$ for each large $i\in I_6$.  The first equality in \eqref{NVD3D6:propD6:eq1} then gives
\begin{equation}
 \label{NVD3D6:propD6:eq2}
  F_i^2 \asymp |T_iS_i^2V_i|.
\end{equation}
The equalities \eqref{NVD3D6:propD6:eq1} also imply that $F_i-R_i$ divides $135S_i^4V_i^2$ and therefore
\[
  |F_i| \asymp |F_i-R_i| \le 135 |S_i^4V_i^2|.
\]
Substituting this upper bound for $|F_i|$ into \eqref{NVD3D6:propD6:eq2}, we obtain
\[
  |T_i| \ll |S_i^6V_i^3|.
\]
Viewing this as a lower bound for $|S_i^2V_i|$, we also deduce from \eqref{NVD3D6:propD6:eq2} that
\[
  |T_i|^{4/3} \ll F_i^2.
\]
Let $j$ denote the successor of $i$ in $I$.  Applying the usual estimates for $|F_i|$, $|S_i|$, $|T_i|$ and $|V_i|=|S_j|$, as in the proof of Corollary \ref{NVF:cor:oT}, the last two inequalities yield
\begin{align*}
 1 &\ll |T_i|^{-1}|S_i^6V_i^3|
   \asymp \frac{X_i}{X_j}|S_i|^5|V_i|^3
   \ll \frac{X_i}{X_j}(X_i^2L_i)^5(X_j^2L_j)^3
   = X_i^{11} X_j^5 L_i^5 L_j^3, \\
 1 &\ll |T_i|^{-1}|F_i|^{3/2}
   \asymp \frac{X_i}{X_j |S_i|} |F_i|^{3/2}
  \ll \frac{X_i}{X_j}(X_j^2L_i^4)^{3/2}
  = X_i X_j^2 L_i^6.
\end{align*}
Then, using $L_i\ll X_{i+1}^{-\lambda}$, $L_j\ll X_{j+1}^{-\lambda}$ and $X_j\ll (X_{i+1}X_{j+1})^{1-\lambda}$ to eliminate $L_i$, $L_j$ and $X_j$ from the above upper bounds, we obtain
\[
 1 \ll X_i^{11} X_{i+1}^{5-10\lambda} X_{j+1}^{5-8\lambda}
 \et
 1 \ll X_i X_{i+1}^{2-8\lambda} X_{j+1}^{2-2\lambda}.
\]
For $i$ large enough, we also have $F_h\neq 0$ where $h$ denotes the predecessor of $i$ of $I$, and so Lemma \ref{F:lemma:tech_alpha_beta} gives $X_i \ll X_{i+1}^\alpha$ where $\alpha = 2\lambda(1-\lambda)/(3\lambda-1)$.  Applying this to eliminate $X_i$ from the preceding estimates, we obtain
\[
 X_{j+1}^{8\lambda-5} \ll X_{i+1}^{11\alpha+5-10\lambda}
 \et
 X_{i+1}^{8\lambda-2-\alpha} \ll X_{j+1}^{2-2\lambda}
\]
and so $(8\lambda-2-\alpha)(8\lambda-5) \le (11\alpha+5-10\lambda)(2-2\lambda)$.  This is the required contradiction.
\end{proof}

In view of the above result, we may rewrite Proposition \ref{NVD2:prop:D3D6} as follows.

\begin{cor}
 \label{NVD3D6:cor}
Suppose that $\lambda\ge 0.7034$. Then, we have $|q_i|\ll |F_i|$ and $|q_i|^6 \ll |T_iF_i^4|$ for each $i\in I$ large enough so that $T_i\neq 0$ and $F_i\neq 0$.
\end{cor}

\begin{proof}[Proof of Theorem \ref{intro:thm:main}]
Suppose on the contrary that $\lambda >\mu=2(9+\sqrt{11})/35$.  Applying the usual estimates, as in the proof of Corollary \ref{NVD2:thm:cor}, the above corollary yields, for each large enough $i\in I$,
\[
 1 \ll |q_i|^{-6}|T_iF_i^4|
   \ll X_i X_{i+1}^{6-17\lambda} X_j^3
   \ll X_{h+1}^{1-\lambda} X_{i+1}^{10-21\lambda} X_{j+1}^{3(1-\lambda)}
\]
where $h$ denotes the predecessor of $i$ in $I$ and $j$ its successor.  We also have
\[
 1 \le |S_j|
   \ll X_j^2X_{j+1}^{-\lambda}
   \ll X_{i+1}^{2(1-\lambda)} X_{j+1}^{2-3\lambda}.
\]
Since $\lambda>\mu$, these estimates remain valid if we substitute $\mu$ for $\lambda$ in both of them.  More precisely, for each large enough $i\in I$, we have
\[
 1 \le X_{h+1}^{1-\mu} X_{i+1}^{10-21\mu} X_{j+1}^{3(1-\mu)}
 \et
 1 \le X_{i+1}^{2(1-\mu)}(2X_{j+1})^{2-3\mu}.
\]
Put
\[
 \nu = \frac{2(1-\mu)}{3\mu-2} = 2+\sqrt{11}.
\]
Then, the last two inequalities are respectively equivalent to
\[
 \left(\frac{X_{i+1}^\nu}{X_{j+1}}\right)^{3\nu} \le \frac{X_{h+1}^\nu}{X_{i+1}}
 \et
 2 \le \frac{X_{i+1}^\nu}{X_{j+1}}\,.
\]
This is impossible since the second inequality shows that, for all large enough pairs of consecutive elements $i<j$ of $I$, the ratios $X_{i+1}^\nu/X_{j+1}$ are bounded below by $2$, while the first implies that they decrease to $1$ as $i$ goes to infinity in $I$.
\end{proof}

%
%

\section{A general family of auxiliary polynomials}
\label{sec:pol}

Let $d\in \bN$.  For each non-empty subset $E$ of the set
\[
 \cT_d:=\{ (m,n)\in\bN^2 \,;\, 2m+3n\le d \},
\]
we choose a non-zero polynomial $P_E$ of $\cR_d$ of the form
\[
 P_E = \sum_{(m,n)\in E} a_{m,n} T^{d-2m-3n}F^m(S^2V)^n
\]
whose $J$-valuation is maximal (i.e.\ which lies in $J^{(\ell)}$ for a largest possible $\ell\in \bN$).  We will not need its precise $J$-valuation but just the fact that, by Theorem \ref{search:thm:dim}, we have $P_E\in J^{(k+1)}$ if $|E|>\tau(k)$ for some integer $k\ge -1$.  We denote by $\cP_d$ the finite set of all polynomials $P_E$ as $E$ runs through the non-empty subsets of $\cT_d$.  The goal of this section is to prove the following result.

\begin{theorem}
\label{pol:thm}
Suppose that $\lambda >\lambda_0:=(1+3\sqrt{5})/11 \cong 0.7007$.  Then there exists a positive integer $d$ and a polynomial $P\in\cP_d$ such that $P(\ux_i,\ux_j)=0$ for infinitely many pairs of consecutive elements $i<j$ in $I$.
\end{theorem}

We start with two simple lemmas.

\begin{lemma}
\label{pol:lemma:E}
Let $a,b,c>0$ and $E=\{(m,n)\in\bN^2 \,;\, am+bn\le c \}$.  Then,
\[
 \frac{c^2}{2ab} \le |E| \le \frac{(a+b+c)^2}{2ab}.
\]
In particular, for each $d\in \bN$, we have
\[
 \frac{d^2}{12} \le \tau(d)=|\cT_d| \le \frac{(d+5)^2}{12}.
\]
\end{lemma}

\begin{proof}
For each $z\in\bR$, set $\cT(z):=\{(x,y)\in\bR^2\,;\, x,y\ge 0 \text{ and } ax+by\le z \}$.  Then,
\[
 \cT(c) \subseteq \bigcup_{(m,n)\in E} \Big( (m,n)+[0,1]^2 \Big) \subseteq \cT(a+b+c)
\]
and the estimates for $|E|$ follow by computing the area of the three regions.  The estimates for $\tau(d)$ correspond to the choice of parameters $a=2$, $b=3$, $c=d$, since for these we have $|E|=\tau(d)$.
\end{proof}

\begin{lemma}
\label{pol:lemma:est}
Suppose that $\lambda>(5-\sqrt{13})/2\cong 0.697$.  For each sufficiently large element $i$ of $I$, we have
\[
 1\le |F_iS_i^2V_i| \le |T_i|
 \et
 |q_i|\le |T_i|\le |q_i|^3.
\]
\end{lemma}

\begin{proof}
The first series of inequalities is a direct consequence of Corollary \ref{NVF:cor:oT} together with the fact that $F_i$, $S_i$ and $V_i$ do not vanish for each sufficiently large $i$ (by Theorems \ref{prelim:thm:phi} and \ref{NVF:thm}).  For the second series, we first recall that, by Proposition \ref{prelim:prop:est_Phi}, we have
$|T_i| \asymp X_iX_j\delta_i$ for each pair of consecutive elements $i<j$ in $I$ with $i$ large enough.  Combining this with the estimates $X_i^{-2} \ll \delta_i \ll X_{i+1}^{-\lambda}$ of Corollary \ref{prelim:cor:est_varphi}, and $X_i\ll X_{i+1}^\alpha$ of Lemma \ref{F:lemma:tech_alpha_beta}, we obtain
\[
 X_{i+1}^{-\alpha} X_j \ll |T_i| \ll X_{i+1}^{\alpha-\lambda} X_j.
\]
Since $|q_i|\asymp X_j/X_{i+1}$ and $\alpha < 0.4$, this yields $|q_i| =o(|T_i|)$.  Using the estimate $X_j\gg X_{i+1}^{2\lambda}$ from Proposition \ref{F:prop:Fnotzero}, it also yields
\[
 |T_i| \ll |q_i|^3 X_{i+1}^{3.4-\lambda} X_j^{-2}
       \ll |q_i|^3 X_{i+1}^{3.4-5\lambda}
       =o(|q_i|^3).
\qedhere
\]
\end{proof}

The next result provides the main tool in the proof of Theorem \ref{pol:thm}.

\begin{prop}
\label{pol:prop}
Suppose that $\lambda>(5-\sqrt{13})/2\cong 0.697$ and let $\epsilon>0$.  Then, there exist integers $d=d(\epsilon)$ and $i_0=i_0(\epsilon)$ with the following property.  For each pair of consecutive elements $i<j$ in $I$ with $i\ge i_0$ such that
\begin{equation}
\label{pol:prop:eq1}
 \log\left|\frac{T_i^2}{F_i}\right| \log\left|\frac{T_i^3}{S_i^2V_i}\right|
 \ge
 6(\log|T_i|)^2 - (6-\epsilon)(\log|q_i|)^2,
\end{equation}
there exists a polynomial $P\in \cP_d$ such that $P(\ux_i,\ux_j)=0$.
\end{prop}

\begin{proof}
Fix $i\in I$ satisfying \eqref{pol:prop:eq1}, with $i$ large enough so that we have $|q_i|\ge 3$ and that all the inequalities of Lemma \ref{pol:lemma:est} are satisfied.  Then, none of the integers $F_i$, $S_i$, $T_i$, $V_i$ is zero, and we may define
\[
 f=\frac{\log|F_i|}{\log|q_i|},
 \quad
 s=\frac{\log|S_i^2V_i|}{\log|q_i|},
 \quad
 t=\frac{\log|T_i|}{\log|q_i|}
 \et
 \sigma=(2t-f)(3t-s).
\]
More precisely, we have $|F_i|\ge 1$, $|S_i^2V_i| \ge 1$, and so Lemma \ref{pol:lemma:est} yields
\begin{equation}
\label{pol:prop:eqfst}
 0 \le f \le t, \quad 0 \le s\le t \et 1 \le t\le 3,
\end{equation}
while the hypothesis \eqref{pol:prop:eq1} becomes
\begin{equation}
\label{pol:prop:eqsigma}
 \sigma \ge 6t^2-(6-\epsilon).
\end{equation}
Let $d\in\bN^*$ and let $E$ denote the set of points $(m,n)\in\cT_d$ satisfying
\begin{equation}
\label{pol:prop:eq2}
 \left| T_i^{d-2m-3n} F_i^m (S_i^2V_i)^n \right| \le |q_i|^k,
 \quad\text{with}\quad
 k:=\left[\frac{6td}{\sigma+6}\right].
\end{equation}
We claim that we can choose $d$ so that $|E| > \tau(k)$ when $i$ is large enough.  If we take this for granted, then we have $P_E\in J^{(k+1)}$ and so $q_i^{k+1}$ divides $P_E(\ux_i,\ux_j)$, where $j$ denotes the successor of $i$ in $I$. On the other hand, the conditions \eqref{pol:prop:eq2} imply that $|P_E(\ux_i,\ux_j)| \le c(d) |q_i|^k$ with a constant $c(d)$ depending only on $d$.  Thus, for $i$ large enough, we have $|P_E(\ux_i,\ux_j)| < |q_i|^{k+1}$ and so $P_E(\ux_i,\ux_j)=0$, as requested by the theorem.

To prove the claim,  we first note that the condition \eqref{pol:prop:eq2} is equivalent to
\[
 (d-2m-3n)t+mf+ns \le k \ssi (2t-f)m +(3t-s)n \ge dt-k.
\]
Applying Lemma \ref{pol:lemma:E} with $a=2t-f$, $b=3t-s$, $c=dt-k$, and noting that these values of $a$, $b$, $c$ are positive in view of \eqref{pol:prop:eqfst} and \eqref{pol:prop:eq2}, we find that the set of points $(m,n)\in\bN^2$ which do not satisfy this condition has cardinality at most $(dt-k+5t)^2/(2\sigma)$.  So we obtain
\[
 |E| \ge |\cT_d|-\frac{(dt-k+5t)^2}{2\sigma}.
\]
The same lemma also gives $|\cT_d| = \tau(d) \ge d^2/12$ and $\tau(k) \le (k+5)^2/12$.  Therefore the condition $|E| > \tau(k)$ is fulfilled if
\begin{equation}
\label{pol:prop:eq3}
 \frac{d^2}{12}-\frac{(dt-k+5t)^2}{2\sigma} > \frac{(k+5)^2}{12}.
\end{equation}
We need to show that this inequality holds as soon as $d$ is large enough, independently of the values of $f$, $s$, $t$ and $\sigma$ satisfying \eqref{pol:prop:eqfst} and \eqref{pol:prop:eqsigma}, as these depend on $i$.  By \eqref{pol:prop:eqfst}, we have $2t^2\le \sigma\le 6t^2$ and $1\le t\le 3$.  Thus, \eqref{pol:prop:eq3} holds if
\[
 d^2 - \frac{6}{\sigma}\left(td-\frac{6td}{\sigma+6}\right)^2 -\left(\frac{6td}{\sigma+6}\right)^2
 \ge c_1d+c_2,
\]
for some absolute constants $c_1$ and $c_2$.  After simplifications, this becomes
\[
 \left(\frac{\sigma-6t^2+6}{\sigma+6}\right) d^2 \ge c_1d+c_2.
\]
Using \eqref{pol:prop:eqsigma} and the crude estimate $\sigma+6\le 6t^2+6\le 60$, we find that the latter inequality holds if $(\epsilon/60) d^2 \ge c_1d+c_2$.  So it holds as soon as $d$ is sufficiently large in terms of $\epsilon$ only.
\end{proof}

\begin{proof}[Proof of Theorem \ref{pol:thm}]
Suppose on the contrary that no such pair $(d,P)$ exists.  Then, for each integer $d\ge 1$, there are only finitely many pairs of consecutive elements $i<j$ of $I$ for which $(\ux_i,\ux_j)$ is a zero of at least one of the polynomials in the finite set $\cP_d$. Therefore, the preceding proposition shows that, for each $\epsilon >0$ and each $i\in I$ with $i\ge i_0(\epsilon)$, we have
\begin{equation}
 \label{pol:thm:eq:main}
 \log\left|\frac{T_i^2}{F_i}\right| \log\left|\frac{T_i^3}{S_i^2V_i}\right|
 < 6(\log|T_i|)^2 - (6-\epsilon)(\log|q_i|)^2.
\end{equation}
Our first goal is to replace this condition by an inequality involving only $\log X_{h+1}$, $\log X_{i+1}$ and $\log X_{j+1}$ where, as usual, $h$ denotes the element of $I$ that comes immediately before $i$, and $j$ the one that comes immediately after.  To this end, we first note that the hypothesis $\lambda>\lambda_0=(1+3\sqrt{5})/11$ implies that $L_i = o(X_{i+1}^{-\lambda_0})$ and $L_j = o(X_{j+1}^{-\lambda_0})$.  So, for $i$ is large enough, we have
\begin{align}
 &|T_i| \le \oT_i:=X_iX_jX_{i+1}^{-\lambda_0}
   &&\text{by Proposition \ref{prelim:prop:est_Phi}},
     \label{pol:thm:eq:majT}\\
 &|F_i| \le \oF_i:=X_j^2X_{i+1}^{-4\lambda_0}
   &&\text{by Proposition \ref{F:prop:Fnotzero}},
     \label{pol:thm:eq:majF}\\
 &|S_i^2V_i| \le \oG_i:=(X_i^2X_{i+1}^{-\lambda_0})^2(X_j^2X_{j+1}^{-\lambda_0})
   &&\text{by Corollary \ref{prelim:cor:est_varphi}},
     \label{pol:thm:eq:majSV}\\
 &X_i \le X_{h+1}^{1-\lambda}X_{i+1}^{1-\lambda_0},
 \quad
  X_j \le (X_{i+1}X_{j+1})^{1-\lambda_0}
   &&\text{by Proposition \ref{prelim:prop:height}}.
     \label{pol:thm:eq:majX}
\end{align}
Moreover, the estimates of Lemma \ref{F:lemma:tech_alpha_beta} combined with $|q_i| \asymp X_j/X_{i+1}$  from Proposition \ref{search:prop:qi} imply that
\[
 \log |q_i| =\log(X_j/X_{i+1})+\cO(1) \asymp \log X_{i+1} \asymp \log X_{h+1}.
\]
So, there is a constant $c>0$, independent of the choice of $\epsilon>0$, such that, for $i$ large enough, we have
\begin{equation}
 \label{pol:thm:eq:minq}
 (6-\epsilon)(\log |q_i|)^2 \ge 6(\log(X_j/X_{i+1}))^2-c\epsilon(\log X_{h+1})^2.
\end{equation}
Assume that $i$ is large enough so that \eqref{pol:thm:eq:main}--\eqref{pol:thm:eq:minq} hold. If we subtract the left hand side of \eqref{pol:thm:eq:main} from its right hand side and expand the resulting expression as a polynomial in $\log |T_i|$, we find a linear polynomial whose coefficient of $\log |T_i|$ is $\log |F_i^3(S_i^2V_i)^2| \ge 0$.  Therefore \eqref{pol:thm:eq:main} remains true if we replace everywhere $|T_i|$ by the upper bound $\oT_i$ given by \eqref{pol:thm:eq:majT}.  By \eqref{pol:thm:eq:majF} and \eqref{pol:thm:eq:majSV}, we also have
\begin{equation}
 \label{pol:thm:eq:factors}
 \frac{\oT_i^2}{|F_i|} \ge \frac{\oT_i^2}{\oF_i} = X_i^2X_{i+1}^{2\lambda_0}
 \et
 \frac{\oT_i^3}{|S_i^2V_i|} \ge \frac{\oT_i^3}{\oG_i} = X_i^{-1}X_{i+1}^{-\lambda_0}X_jX_{j+1}^{\lambda_0}.
\end{equation}
As both of these lower bounds are greater than $1$, they have a positive logarithm. So, we may further replace $|F_i|$ by $\oF_i$ and $|S_i^2V_i|$ by $\oG_i$, and thus
\[
 \log\frac{\oT_i^2}{\oF_i} \log\frac{\oT_i^3}{\oG_i}
 < 6(\log\oT_i)^2 -(6-\epsilon)(\log |q_i|)^2.
\]
Using \eqref{pol:thm:eq:majF}, \eqref{pol:thm:eq:minq} and \eqref{pol:thm:eq:factors}, this yields
\[
 \begin{aligned}
 0
  &< -\log(X_i^2X_{i+1}^{2\lambda_0})\log(X_i^{-1}X_{i+1}^{-\lambda_0}X_jX_{j+1}^{\lambda_0}) \\
     &\qquad + 6 \big(\log(X_iX_jX_{i+1}^{-\lambda_0})\big)^2
             - 6 \big(\log(X_jX_{i+1}^{-1})\big)^2 + c\epsilon (\log X_{h+1})^2 \\
  &= 2(\log X_i)\log \big( X_i^4 X_{i+1}^{-4\lambda_0} X_j^5 X_{j+1}^{-\lambda_0} \big) \\
     &\qquad + 2(\log X_{i+1})\log \big( X_{i+1}^{4\lambda_0^2-3} X_j^{6-7\lambda_0} X_{j+1}^{-\lambda_0^2} \big)
             + c\epsilon (\log X_{h+1})^2.
 \end{aligned}
\]
Note that, in this last expression, the first product is positive for $i$ large enough because, using $X_i\gg X_{i+1}^{\lambda_0/2}$ and $X_j\gg X_{j+1}^{\lambda_0/2}$ (Lemma \ref{F:lemma:tech_alpha_beta}), we find that
\[
 X_i^4 X_{i+1}^{-4\lambda_0} X_j^5 X_{j+1}^{-\lambda_0}
  \gg X_{i+1}^{-2\lambda_0} X_j^3
  \ge X_j^{3-2\lambda_0}
\]
tends to infinity with $i$.  This allows us to use \eqref{pol:thm:eq:majX} to eliminate both $X_i$ and $X_j$.  After simplifications, this yields
\begin{equation}
 \label{pol:thm:eq:mainbis}
 \begin{aligned}
 0 < &(4(1-\lambda)^2+c\epsilon/2)(\log X_{h+1})^2 \\
     &+ (1-\lambda)(13-17\lambda_0)(\log X_{h+1})(\log X_{i+1})\\
     &+ (1-\lambda)(5-6\lambda_0)(\log X_{h+1})(\log X_{j+1})\\
     &+ (12-35\lambda_0+24\lambda_0^2)(\log X_{i+1})^2\\
     &+ (11-24\lambda_0+12\lambda_0^2)(\log X_{i+1})(\log X_{j+1})
 \end{aligned}
\end{equation}
Now, put
\[
 \epsilon=\frac{8(1-\lambda_0)^2-8(1-\lambda)^2}{c},\quad
 \rho_i=\frac{\log X_{i+1}}{\log X_{h+1}}, \quad
 \rho_j=\frac{\log X_{j+1}}{\log X_{i+1}},
\]
thus fixing the choice of $\epsilon>0$.  We substitute this value of $\epsilon$ into \eqref{pol:thm:eq:mainbis} and note that the resulting inequality remains valid if we replace $\lambda$ by $\lambda_0$.  After dividing both sides by $(\log X_{h+1})^2$, it yields
\begin{equation}
 \label{pol:thm:eq:mainter}
 \begin{aligned}
 0 < 4(1-\lambda_0)^2
          &+ (1-\lambda_0)(13-17\lambda_0)\rho_i + (1-\lambda_0)(5-6\lambda_0)\rho_i\rho_j\\
          &+ (12-35\lambda_0+24\lambda_0^2)\rho_i^2 + (11-24\lambda_0+12\lambda_0^2)\rho_i^2\rho_j.
 \end{aligned}
\end{equation}

Suppose that there are arbitrarily large pairs of consecutive elements $i<j$ in $I$ with $\rho_j<\rho_i$.  Then, \eqref{pol:thm:eq:mainter} holds with $\rho_j$ replaced by $\rho_i$ because in the right hand side of this inequality all terms involving $\rho_j$ have positive coefficients.  This means that
\begin{equation}
 \label{pol:thm:eq:rho}
 \begin{aligned}
 0 < 4(1-\lambda_0)^2
          &+ (1-\lambda_0)(13-17\lambda_0)\rho_i\\
          &+ (17-46\lambda_0+30\lambda_0^2)\rho_i^2 + (11-24\lambda_0+12\lambda_0^2)\rho_i^3.
 \end{aligned}
\end{equation}
On the other hand, Lemma \ref{F:lemma:tech_alpha_beta} gives $\rho_i \le \beta+o(1)$ with $\beta=2(1-\lambda)/(3\lambda-2)$ and so, for $i$ large enough, we have
\[
 \rho_i < \beta_0:=\frac{2(1-\lambda_0)}{3\lambda_0-2}=\frac{5+3\sqrt{5}}{2}.
\]
This is a contradiction because \eqref{pol:thm:eq:rho} can be rewritten in the form
\begin{equation}
 \label{pol:thm:eq:factor}
 0 < (\rho_i-\beta_0)(a\rho_i^2-b\rho_i-c)
 \quad
 \text{with $a>b>c>0$}
\end{equation}
while it follows from Corollary \ref{prelim:cor:XjLj} that $\rho_i\ge 2+o(1)$ and so $a\rho_i^2-b\rho_i-c>0$ for each sufficiently large $i$.

Thus we have $\rho_i\le \rho_j$ for each sufficiently large pairs of consecutive elements $i<j$ in $I$.  Then $(\rho_i)_{i\in I}$ converges to a limit $\rho$ with $2\le \rho\le \beta < \beta_0$ and, by continuity, the inequality \eqref{pol:thm:eq:mainter} holds with $\rho_i$ and $\rho_j$ replaced by $\rho$.  So \eqref{pol:thm:eq:factor} holds with $\rho_i$ replaced by $\rho$.  Again, this is impossible.
\end{proof}

\begin{remark}
\label{pol:remark:rel_prime}
As the proof of Proposition \ref{pol:prop} shows, when \eqref{pol:prop:eq1} holds, we get several polynomials $P$ satisfying $P(\ux_i,\ux_j)=0$ for the same pairs $(i,j)$ by varying the integer $d$.  If we could make these relatively prime as a set, this would contradict Proposition \ref{search:prop:remark}, meaning that \eqref{pol:thm:eq:main} holds for any given $\epsilon>0$ and any sufficiently large $i\in I$.  Then, the above argument would yield $\lambda \le (1+3\sqrt{5})/11$.
\end{remark}

\begin{remark}
Assuming that $\lambda = (1+3\sqrt{5})/11$, we also note that \eqref{pol:thm:eq:main} holds for any given $\epsilon>0$ and any sufficiently large $i\in I$ when
\[
 X_j \asymp X_i^{(5+3\sqrt{5})/2}, \quad
 X_{i+1} \asymp X_i^{(3\sqrt{5}-1)/2}, \quad
 L_i \asymp X_{i+1}^{-\lambda} \asymp X_i^{-2},
\]
where $j$ stands for the successor of $i$ in $I$.  Then, one finds that $|S_i|\asymp |V_i|\asymp 1$, $|T_i| \asymp X_i^{3(1+\sqrt{5})/2}$, $|F_i|\asymp X_i^{3(\sqrt{5}-1)}$ and $|q_i|\asymp X_i^3$.
\end{remark}

%
%

\section{A special family of auxiliary polynomials}
\label{sec:special}

In view of Remark \ref{pol:remark:rel_prime} above, we would get $\lambda \le (1+3\sqrt{5})/11$ if we could prove for example that the polynomials of $\cP_d$ are irreducible for arbitrarily large values of $d$.  This is probably too much to ask.  Nevertheless, this can be done for certain polynomials of the type $P_E$ as the next result illustrates.

\begin{theorem}
\label{special:thm}
Let $d=12\ell+2$ for some $\ell\in\bN$, and let
\[
 E_\ell=\left\{ (m,n)\in\bN^2\,;\, 2m+3n\le d, \ \frac{m}{6\ell+1}+\frac{n}{3\ell}\ge 1 \right\}.
\]
Then the set of polynomials of $\cR_d \cap J^{(6\ell+2)}$ of the form
\begin{equation}
 \label{special:thm:eq:P}
 \sum_{(m,n)\in E_\ell} a_{m,n} T^{d-2m-3n}F^m(S^2V)^n
\end{equation}
is a one-dimensional vector space generated by an irreducible polynomial $P_\ell$ of $\cR$ which has $a_{6\ell+1,0}\neq 0$ and $a_{0,3\ell}\neq 0$.
\end{theorem}

Before going into its proof, we deduce the following consequence.

\begin{cor}
 \label{special:cor}
Suppose that $\lambda>2/3$, and let $\epsilon>0$.  For each $i\in I$ large enough so that $S_iT_iV_i\neq 0$, we have $|F_i|\gg |q_i|$ or $|T_iS_i^2V_i|\gg |q_i|^{2-\epsilon}$ with implied constants depending only on $\xi$ and $\epsilon$.
\end{cor}

\begin{proof}
Each $E_\ell$ is the set of integral points in the triangle with vertices $(6\ell+1,0)$, $(0,3\ell)$ and $(0,4\ell+(2/3))$. So, for consecutive $i<j$ in $I$, we have
\begin{align*}
 |P_\ell(\ux_i,\ux_j)|
   &\ll_{\xi,P} |F_i|^{6\ell+1} + |T_i|^{3\ell+2}|S_i^2V_i|^{3\ell} + |S_i^2V_i|^{4\ell+(2/3)}\\
   &\ll |F_i|^{6\ell+1} + |T_iS_i^2V_i|^{3\ell+2}
\end{align*}
where the second inequality uses the estimate $|S_i^2V_i|=o(|T_i|^3)$ from Lemma \ref{search:lemma:oTcube}.
If $P_\ell(\ux_i,\ux_j)\neq 0$, Proposition \ref{search:prop:divJk} also gives $|q_i|^{6\ell+2}\le |P_\ell(\ux_i,\ux_j)|$, thus
\begin{equation}
 \label{special:cor:eq}
 |F_i| \gg_{\xi,\ell} |q_i|^{(6\ell+2)/(6\ell+1)} \ge |q_i|
 \quad\text{or}\quad
 |T_iS_i^2V_i| \gg_{\xi,\ell} |q_i|^{(6\ell+2)/(3\ell+2)}.
\end{equation}
Choose $\ell$ to be the smallest positive integer such that $(6\ell+2)/(3\ell+2) \ge 2-\epsilon$. Since $P_\ell$ and $P_{\ell+1}$ are irreducible of distinct degrees, they are relatively prime and Proposition \ref{search:prop:remark} shows that at least one of them does not vanish at the point $(\ux_i,\ux_j)$ for $i$ sufficiently large.  Then \eqref{special:cor:eq} holds with the given value of $\ell$ or with $\ell$ replaced by $\ell+1$, and the result follows.
\end{proof}

As a first step towards the proof of Theorem \ref{special:thm}, we first note that the irreducibility of $P_\ell$ derives simply from the non-vanishing of its coefficients of indices $(6\ell+1,0)$ and $(0,3\ell)$.

\begin{lemma}
\label{special:lemma:irred}
With the notation of Theorem \ref{special:thm}, let $P$ be a non-zero polynomial of the form \eqref{special:thm:eq:P} with non-zero coefficients of indices $(6\ell+1,0)$ and $(0,3\ell)$.  Then $P$ is an irreducible element of $\cR$.
\end{lemma}

\begin{proof}
Suppose on the contrary that $P$ is not irreducible.  Then it factors as a product $P=AB$ where $A$ and $B$ are homogeneous elements of $\cR$ of degree less than $d$.  For the present purpose, we define the index of each monomial $T^kF^m(S^2V)^n$ with $(k,m,n)\in\bN^3$ as $\imath(m,n)=m/(6\ell+1)+n/(3\ell)$.  We also define the index $\imath(Q)$ of an arbitrary non-zero element $Q$ of $\cR$ as the \emph{smallest} index of its monomials.  Then, we have $\imath(P)=\imath(A)+\imath(B)$ and the homogeneous part $P_0$ of $P$ of smallest index $\imath(P)$ is the product of the homogeneous parts $A_0$ of $A$ and $B_0$ of $B$ with smallest index.  However, the function $\imath$ is injective on the set of pairs $(m,n)\in\bN^2$ with $2m+3n<12\ell+2$ because its kernel on $\bZ^2$ is the subgroup generated by $(6\ell+1,-3\ell)$.  Thus $A_0$ and $B_0$ are monomials.  This is impossible because the hypothesis implies that $\imath(P)=1$ with $P_0$ involving the monomials associated with $(6\ell+1,0)$ and $(0,3\ell)$.
\end{proof}

The polynomials $F$, $G:=TS^2V$, $H:=S^4V^2$ are homogeneous elements of $\cR$ of respective degrees $2$, $4$, $6$.  They generate a graded subalgebra of $\cR$,
\[
 \cS:=\bQ[F,G,H]=\bigoplus_{\ell=0}^\infty \cS_{2\ell},
\]
where $\cS_{2\ell}=\cS\cap\cR_{2\ell}$ admits, as a basis, the products $F^kG^mH^n$ with  $(k,m,n)\in\bN^3$ satisfying $k+2m+3n=\ell$.  In particular, we have
\begin{equation}
 \label{special:eq:dimS}
 \dim_\bQ\cS_{2\ell}=\tau(\ell).
\end{equation}
Moreover, the formulas \eqref{search:eq:MN} show that $\cS=\bQ[F,M,N]$.  As $F$, $M$ and $N$ are homogeneous of respective degrees $2$, $4$ and $6$, it follows that the products $F^kM^mN^n$ with $(k,m,n)\in\bN^3$, $k+2m+3n=\ell$, form another basis of $\cS_{2\ell}$ over $\bQ$.  The connection with the current situation is the following.

\begin{lemma}
\label{specvail:lemma:decomp}
With the notation of Theorem \ref{special:thm}, the set of polynomials of the form \eqref{special:thm:eq:P} constitutes the vector space $\cV_\ell:=\cS_{12\ell+2} \oplus \langle G^{3\ell}T^2 \rangle_\bQ$.
\end{lemma}

\begin{proof}
For any $\ell\in \bN$, the vector space $\cS_{2\ell}$ is generated by the products $T^kF^m(S^2V)^n$ with $(k,m,n)\in\bN^3$ satisfying $k+2m+3n=2\ell$ and $k\le n$.  So, equivalently, it is generated by the products $T^{2\ell-2m-3n}F^m(S^2V)^n$ with $(m,n)\in\bN^2$ satisfying $2m+3n\le 2\ell$ and $\ell \le m+2n$.  Let $\cT_{2\ell}^*$ denote this subset of $\bN^2$.  Then, the conclusion follows by observing that we have $E_\ell = \cT_{12\ell+2}^* \cup \{(0,3\ell)\}$.
\end{proof}

The proof of the next result is very similar to that of Theorem \ref{search:thm:dim}, based on the fact that $M$ and $N$ have respective $J$-valuations $2$ and $3$, so we omit its proof.

\begin{lemma}
\label{special:lemma:dimS}
For each choice of integers $k,\ell\ge 0$, we have
\[
 \dim_\bQ \cS_{2\ell} = \tau(\ell)
 \et
 \dim_\bQ \frac{\cS_{2\ell}}{\cS_{2\ell}\cap J^{(k)}}
  = \begin{cases}
     \tau(k-1) &\text{if $k\le \ell$,}\\
     \tau(\ell) &\text{if $k > \ell$.}
    \end{cases}
\]
Moreover, if $k\le \ell$, a basis of $\cS_{2\ell}\cap J^{(k)}$ over $\bQ$ is given by the products $F^{\ell-2m-3n}M^mN^n$ with $(m,n)\in\bN^2$ satisfying $k\le 2m+3n\le \ell$.
\end{lemma}

\begin{proof}[Proof of Theorem \ref{special:thm}]
In view of Lemmas \ref{special:lemma:irred} and \ref{specvail:lemma:decomp}, we need to show that $\cV_\ell \cap J^{(6\ell+2)}$ is one-dimensional generated by a polynomial $P_\ell$ whose coefficients of $F^{6\ell+1}$ and $G^{3\ell}T^2$ are both non-zero.

By the formulas of Lemma \ref{special:lemma:dimS}, we have $\cS_{2\ell}\cap J^{(\ell+1)}=\{0\}$ for each $\ell \ge 0$.  In particular, this gives $\cS_{12\ell+2}\cap J^{(6\ell+2)}=\{0\}$, and so $\dim_\bQ(\cV_\ell\cap J^{(6\ell+2)})\le 1$.  Moreover, since $\cV_\ell$ has dimension $\tau(6\ell+1)+1$, Theorem \ref{search:thm:dim} implies that it contains a non-zero element of $J^{(6\ell+2)}$.  From this, we conclude that $\cV_\ell\cap J^{(6\ell+2)}$ is one-dimensional generated by a polynomial $P_\ell$ outside of $\cS_{12\ell+2}$.  So, the coefficient of $G^{3\ell}T^2$ in $P_\ell$ is non-zero, and it remains to show the same for the coefficient of $F^{6\ell+1}$.

For this purpose, we note that, since $HT^2=G^2$, we have $H\cV_\ell \subseteq \cS_{12\ell+8}$ and thus $HP_\ell$ belongs to $\cS_{12\ell+8} \cap J^{(6\ell+2)}$.  By Lemma \ref{special:lemma:dimS}, the latter vector space
is generated by the products $F^{6\ell+4-2m-3n}M^mN^n$ with $(m,n)\in\bN^2$ satisfying $6\ell+2\le 2m+3n \le 6\ell+4$,  thus
\begin{equation}
\label{special:thm:eq:HP}
 HP_\ell =
 \sum_{k=0}^\ell (r_kFN + s_kF^2M + t_kM^2)M^{3k}N^{2\ell-2k}
\end{equation}
with $r_k,s_k,t_k\in \bQ$ ($0\le k\le \ell$) not all zero.  In order to fix the choice of $P_\ell$ up to multiplication by $\pm 1$, we request that these coefficients are relatively prime integers.  Since $P_\ell \in \cS_{12\ell+2}+\langle G^{3\ell}T^2 \rangle_\bQ$, they must satisfy
\[
 \sum_{k=0}^\ell (r_kFN + s_kF^2M + t_kM^2)M^{3k}N^{2\ell-2k} \equiv a G^{3\ell+2} \mod H,
\]
for some $a\in\bZ$.  A priori, this is a congruence in $\cS$ but we may view it as a congruence in $\bZ[F,G,H]$ because both sides belong to that ring.  Let $(2,H)$ denote the ideal of $\bZ[F,G,H]$ generated by $2$ and $H$.  Since $M$ and $N$ are respectively congruent to $F^2+G$ and $F^3$ modulo $(2,H)$, this yields
\begin{align*}
 \sum_{k=0}^\ell
   \Big(r_kF^{6\ell-6k+4}(F^2+G&)^{3k}
       + s_kF^{6\ell-6k+2}(F^2+G)^{3k+1}\\
       &+ t_kF^{6\ell-6k}(F^2+G)^{3k+2}\Big) \equiv aG^{3\ell+2} \mod (2,H).
\end{align*}
Substituting $G+F^2$ for $G$ in this congruence, it becomes
\begin{align*}
 \sum_{k=0}^\ell
   \Big(r_kF^{6\ell-6k+4}G^{3k}
       &+ s_kF^{6\ell-6k+2}G^{3k+1}\\
       &+ t_kF^{6\ell-6k}G^{3k+2}\Big) \equiv a(G+F^2)^{3\ell+2} \mod (2,H),
\end{align*}
which by comparing coefficients on both sides yields
\[
 r_k\equiv a\binom{3\ell+2}{3k},\quad
 s_k\equiv a\binom{3\ell+2}{3k+1},\quad
 t_k\equiv a\binom{3\ell+2}{3k+2}
 \mod 2.
\]
As $r_k,s_k,t_k$ are not all even, $a$ must be odd and the above congruences determine the parity of all coefficients. We also observe that $M$ and $N$ are respectively congruent to $F^2$ and $F^3+H$ modulo $(2,G)$, and so
\[
 HP_\ell
 \equiv
 \sum_{k=0}^\ell (r_k(F^3+H) + s_kF^3 + t_kF^3)F^{6k+1}(F^3+H)^{2\ell-2k} \mod (2,G).
\]
In particular, the coefficient $b$ of $F^{6\ell+1}$ in $P_\ell$ is an integer with
\[
 b
 \equiv \sum_{k=0}^\ell \big(r_k(2\ell-2k+1) + (s_k+t_k)(2\ell-2k)\big)
 \equiv \sum_{k=0}^\ell r_k
 \mod 2.
\]
On the other hand, it is known that $\sum_{k=0}^\ell \binom{3\ell+2}{3k}=(2^{3\ell+2}-(-1)^\ell)/3$. Thus, $b$ is odd and therefore non-zero.
\end{proof}

%
%


\begin{thebibliography}{9}



\bibitem{DSa}
  H.~Davenport, W.~M.~Schmidt,
  Approximation to real numbers by quadratic irrationals,
  \textit{Acta Arith.\ }\textbf{13} (1967),
  169-176.
%
\bibitem{DSb}
  H.~Davenport, W.~M.~Schmidt,
  Approximation to real numbers by algebraic integers,
  \textit{Acta Arith.\ }\textbf{15} (1969),
  393--416.
%
\bibitem{Ja}
  V.~Jarn\'{\i}k,
  Zum Khintchineschen \"Ubertragungssatz,
  {\it Trudy Tbilisskogo mathematicheskogo instituta im.\
  A.~M.~Razmadze = Travaux de l'Institut math\'ematique
  de Tbilissi} {\bf 3} (1938), 193--212.
%
\bibitem{La}
  M.~Laurent,
  Simultaneous rational approximation to
  the successive powers of a real number,
  {\it Indag.\ Math.\ (N.S.)} {\bf 11} (2003), 45--53.
%
\bibitem{Lo}
  S.~Lozier,
  On simultaneous approximation to a real number and its cube,
  M.Sc.~thesis, U.~of Ottawa, 2010, 86 pp.
%
\bibitem{Rb}
  D.~Roy, Approximation to real numbers by cubic algebraic integers I,
  \textit{Proc.\ London Math.\ Soc.\ }\textbf{88}
  (2004), 42--62.
%
\bibitem{Rc}
  D.~Roy,
  On simultaneous rational approximations to a real number, its square, and its cube,
  \textit{Acta Arith.\ }\textbf{133} (2008), 185--197.
%
\bibitem{Sc}
  W.~M.~Schmidt,
  \textit{Diophantine Approximations and Diophantine equations},
  Lecture Notes in Math., vol.~1467, Sprin\-ger-Verlag, 1991.

\end{thebibliography}
\end{document}